\documentclass{amsart}
\usepackage{amssymb,mathtools}
\usepackage[british]{babel}
\usepackage{enumitem}
\usepackage[latin1]{inputenc}
\usepackage{hyperref}
\usepackage{tikz-cd}

\newtheorem{theorem}{Theorem}
\numberwithin{theorem}{section}

\newtheorem{lemma}[theorem]{Lemma}
\newtheorem{proposition}[theorem]{Proposition}
\theoremstyle{definition}
\newtheorem{definition}[theorem]{Definition}
\newtheorem{remark}[theorem]{Remark}
\newtheorem{example}[theorem]{Example}

\newtheorem*{claim}{Claim}

\tikzset{symbol/.style={draw=none,every to/.append style={edge node={node [sloped, allow upside down, auto=false]{$#1$}}}}}

\newcommand{\rca}{\mathbf{RCA}}
\newcommand{\isigma}{\mathbf{I\Sigma}^0}

\newcommand{\lo}{\mathsf{LO}}
\newcommand{\nat}{\mathsf{Nat}}

\newcommand{\supp}{\operatorname{supp}}
\newcommand{\rng}{\operatorname{rng}}
\newcommand{\emp}{\operatorname{emp}}
\newcommand{\nf}{=_{\operatorname{NF}}}

\title{Ackermann and Goodstein go functorial}
\author[Aguilera, Freund, Rathjen, Weiermann]{{Juan P.}~Aguilera, Anton~Freund, Michael~Rathjen and~Andreas~Weiermann}

\begin{document}

\begin{abstract}
We present variants of Goodstein's theorem that are equivalent to arithmetical comprehension and to arithmetical transfinite recursion, respectively, over a weak base theory. These variants differ from the usual Goodstein theorem in that they (necessarily) entail the existence of complex infinite objects. As part of our proof, we show that the Veblen hierarchy of normal functions on the ordinals is closely related to an extension of the Ackermann function by direct limits.
\end{abstract}


\maketitle

\section{Introduction}

G\"odel's incompleteness theorem tells us that no reasonable framework for the foundation of mathematics can allow us to prove all true statements about the natural numbers. It is an important insight of mathematical logic that incompleteness is more than a theoretical possibility: Several mathematical theorems have been shown to be unprovable in Peano arithmetic and in stronger theories~\cite{kirby-paris82,paris-harrington,simpson85,friedman-robertson-seymour}. The first example is Goodstein's theorem, which we discuss in the following.

The Goodstein sequence $G_{b,m}(0),G_{b,m}(1),\dots$ relative to a given starting value $m\in\mathbb N$ and to a non-decreasing function $b:\mathbb N\to\mathbb N$ with $b(0)\geq 2$ is generated by the following process: Put $G_{b,m}(0)=m$. Now assume that~$G_{b,m}(i)$ is already defined. If we have $G_{b,m}(i)=0$, then we set $G_{b,m}(i+1)=0$ and say that the Goodstein sequence terminates. Otherwise, we consider the hereditary base $b(i)$ representation of $G_{b,m}(i)$ (which is computed recursively, by writing the exponents of the usual base $b(i)$ representation in hereditary base $b(i)$ as well). In order to obtain $G_{b,m}(i+1)$, we increase the base to $b(i+1)$ and subtract one from the resulting number. For example, $G_{b,m}(i)=2\,196$, $b(i)=3$ and $b(i+1)=5$ lead to
\begin{align*}
G_{b,m}(i)=3^{3^{3^0}\cdot 2+3^0}+3^{3^0\cdot 2}\,\leadsto\, G_{b,m}(i+1)=5^{5^{5^0}\cdot 2+5^0}+5^{5^0\cdot 2}-1=48\,828\,149.
\end{align*}
While the values grow fast for small exponents, Goodstein's theorem~\cite{goodstein44} states that they must eventually reach zero, for any starting value and any function~$b$ as above. In fact, Goodstein did more than establish the truth of this result: He showed that his theorem is equivalent to the well foundedness of a certain ordinal~$\varepsilon_0$ (provably in a weak base theory, which does not prove the equivalent statements itself).

Even before Goodstein's work, Gentzen~\cite{gentzen36,gentzen43} had shown that Peano arithmetic cannot prove transfinite induction up to~$\varepsilon_0$. This suggests an independence result, which was apparently stated in the submitted version of Goodstein's paper (see \cite{rathjen-goodstein} for all historical information in this paragraph). However, there is one issue with this independence result, which P.~Bernays pointed out in his role as referee: The above formulation of Goodstein's theorem quantifies over all functions $b:\mathbb N\to\mathbb N$. This means that it cannot be expressed in Peano arithmetic (though an extension by set parameters would suffice). As a consequence of this criticism, Goodstein's published paper states the connection with~$\varepsilon_0$ but not the independence~result.

The issue with Goodstein's independence result disappears if we fix some definable function $b:\mathbb N\to\mathbb N$. By a celebrated result of L.~Kirby and J.~Paris~\cite{kirby-paris82}, Goodstein's theorem for $b(i)=2+i$ remains unprovable in Peano arithmetic (see~\cite{rathjen-goodstein} for a proof based on Goodstein's methods). Let us emphasize that the independent statement of Kirby and Paris is very concrete: It asserts that a certain algorithm (which computes the sequence $G_{b,m}(0),G_{b,m}(1),\dots$) terminates on any input (the starting value $m$). While one may expect that parts of abstract mathematics lie beyond the scope of Peano arithmetic, it is remarkable that the latter fails to prove a concrete statement about the termination of a specific natural algorithm.

We have seen that Kirby and Paris's version of Goodstein's theorem (the one with $b(i)=2+i$) is more concrete than the original formulation (which quantifies over all $b:\mathbb N\to\mathbb N$). In the present paper we will push Goodstein's theorem into the opposite direction: Our version (presented below) is more abstract, since it involves existential as well as universal quantification over infinite objects. While concrete independence remains relevant, our abstract and more general results have important advantages: They yield a very uniform picture, in which variants of Goodstein's theorem are equivalent to well-established set existence principles. Specifically, our approach turns Goodstein's theorem into a topic for reverse mathematics (see~\cite{simpson09} for a comprehensive introduction). In this research programme one proves equivalences and (non-)implications between various theorems about finite objects (represented by natural numbers) and countably infinite collections of such objects (subsets of~$\mathbb N$).

Our variant of Goodstein's theorem is based on a very natural idea: When we compute $G_{b,m}(i+1)$ from $G_{b,m}(i)$, we will also increase the coefficients in our representation, not just the base. Since the base bounds the coefficients, the latter will be increased according to a family $c=\langle c_i\rangle_{i\in\mathbb N}$ of strictly increasing functions
\begin{equation*}
c_i:\{0,\dots,b(i)-1\}\to\{0,\dots,b(i+1)-1\}.
\end{equation*}
We cannot, however, increase the coefficients arbitrarily if we want our Goodstein sequences to terminate. As a simple but important example, we consider the binary representation of numbers below $2^b$, where we view $b$ as the base and the exponents as coefficients. More explicitly, the Goodstein sequence $G_{b,c,m}^{2}(0),G_{b,c,m}^{2}(1),\dots$ relative to an increasing function $b:\mathbb N\to\mathbb N$, a family $c=\langle c_i\rangle_{i\in\mathbb N}$ as above and a starting value $m<2^{b(0)}$ is given by $G_{b,c,m}^{2}(0)=m$ and
\begin{equation*}
G_{b,c,m}^{2}(i+1)=\begin{cases}
0 & \text{if $G_{b,c,m}^{2}(i)=0$},\\[1.5ex]
2^{c_i(n_0)}+\ldots+2^{c_i(n_k)}-1 & \parbox[t]{15em}{if $G_{b,c,m}^{2}(i)=2^{n_0}+\ldots+2^{n_k}$\\ with $b(i)>n_0>\ldots>n_k$.}
\end{cases}
\end{equation*}
Since we will study different representations of numbers, we have included a superscript~$2$ that indicates binary representation. To call $b(i)$ the base is somewhat awkward in the present example, but we will nevertheless keep this terminology. For $b(i)=2+i$, $c_i(n)=n+1$ and $m=3$ the definition yields
\begin{equation*}
G_{b,c,m}^{2}(0)=2^1+2^0,\, G_{b,c,m}^{2}(1)=2^2+2^1-1=2^2+2^0,\,\dots,\, G_{b,c,m}^{2}(i)=2^{i+1}+2^0,\,\dots.
\end{equation*}
Why does this sequence not reach zero? A glance at Figure~\ref{fig:glueing-coeff-change} is illuminating. In both subfigures, the first columns represent the sets $\{0,\dots,b(i)-1\}$ with $b(i)=2+i$. The solid arrows represent the functions $c_i$ that we employ to change coefficients. The last column of each subfigure represents the order that we obtain by glueing along the arrows, or in categorical terms: the direct limit. In the left subfigure we have $c_i(n)=n$, as in the usual Goodstein theorem. The direct limit is isomorphic to the well order~$\mathbb N$. In the right subfigure we have $c_i(n)=n+1$, as in our example of a Goodstein sequence that does not terminate. We will see that such an example can only exist if there is an infinitely descending sequence in the direct limit. In the present case, the latter is isomorphic to $\{p\in\mathbb Z\,|\,p\leq 0\}$.

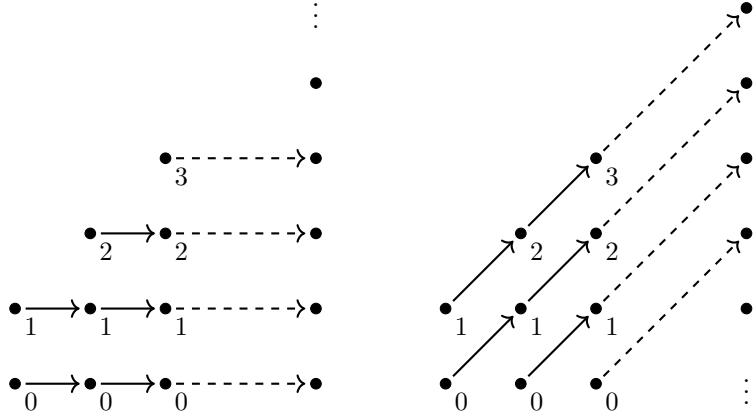
\begin{figure}[h]
\centering
\begin{equation*}
\begin{tikzpicture}
\filldraw[black] (0,0) circle (2pt) node[below right] {0};
\filldraw[black] (0,1) circle (2pt) node[below right] {1};
\filldraw[black] (1,0) circle (2pt) node[below right] {0};
\filldraw[black] (1,1) circle (2pt) node[below right] {1};
\filldraw[black] (1,2) circle (2pt) node[below right] {2};
\filldraw[black] (2,0) circle (2pt) node[below right] {0};
\filldraw[black] (2,1) circle (2pt) node[below right] {1};
\filldraw[black] (2,2) circle (2pt) node[below right] {2};
\filldraw[black] (2,3) circle (2pt) node[below right] {3};
\filldraw[black] (4,0) circle (2pt) node[below right] {};
\filldraw[black] (4,1) circle (2pt) node[below right] {};
\filldraw[black] (4,2) circle (2pt) node[below right] {};
\filldraw[black] (4,3) circle (2pt) node[below right] {};
\filldraw[black] (4,4) circle (2pt) node[below right] {};
\node at (4,5) {$\vdots$};
\draw [shorten >=4pt,shorten <=4pt,->,thick] (0,0) -- (1,0);
\draw [shorten >=4pt,shorten <=4pt,->,thick] (0,1) -- (1,1);
\draw [shorten >=4pt,shorten <=4pt,->,thick] (1,0) -- (2,0);
\draw [shorten >=4pt,shorten <=4pt,->,thick] (1,1) -- (2,1);
\draw [shorten >=4pt,shorten <=4pt,->,thick] (1,2) -- (2,2);
\draw [shorten >=4pt,shorten <=4pt,->,thick,dashed] (2,0) -- (4,0);
\draw [shorten >=4pt,shorten <=4pt,->,thick,dashed] (2,1) -- (4,1);
\draw [shorten >=4pt,shorten <=4pt,->,thick,dashed] (2,2) -- (4,2);
\draw [shorten >=4pt,shorten <=4pt,->,thick,dashed] (2,3) -- (4,3);
\end{tikzpicture}
\qquad\qquad
\begin{tikzpicture}
\filldraw[black] (0,0) circle (2pt) node[below right] {0};
\filldraw[black] (0,1) circle (2pt) node[below right] {1};
\filldraw[black] (1,0) circle (2pt) node[below right] {0};
\filldraw[black] (1,1) circle (2pt) node[below right] {1};
\filldraw[black] (1,2) circle (2pt) node[below right] {2};
\filldraw[black] (2,0) circle (2pt) node[below right] {0};
\filldraw[black] (2,1) circle (2pt) node[below right] {1};
\filldraw[black] (2,2) circle (2pt) node[below right] {2};
\filldraw[black] (2,3) circle (2pt) node[below right] {3};
\node at (4,0) {$\vdots$};
\filldraw[black] (4,1) circle (2pt) node[below right] {};
\filldraw[black] (4,2) circle (2pt) node[below right] {};
\filldraw[black] (4,3) circle (2pt) node[below right] {};
\filldraw[black] (4,4) circle (2pt) node[below right] {};
\filldraw[black] (4,5) circle (2pt) node[below right] {};
\draw [shorten >=4pt,shorten <=4pt,->,thick] (0,0) -- (1,1);
\draw [shorten >=4pt,shorten <=4pt,->,thick] (0,1) -- (1,2);
\draw [shorten >=4pt,shorten <=4pt,->,thick] (1,0) -- (2,1);
\draw [shorten >=4pt,shorten <=4pt,->,thick] (1,1) -- (2,2);
\draw [shorten >=4pt,shorten <=4pt,->,thick] (1,2) -- (2,3);
\draw [shorten >=4pt,shorten <=4pt,->,thick,dashed] (2,0) -- (4,2);
\draw [shorten >=4pt,shorten <=4pt,->,thick,dashed] (2,1) -- (4,3);
\draw [shorten >=4pt,shorten <=4pt,->,thick,dashed] (2,2) -- (4,4);
\draw [shorten >=4pt,shorten <=4pt,->,thick,dashed] (2,3) -- (4,5);
\end{tikzpicture}
\end{equation*}
\caption{Two ways to increase coefficients for representations with base~$b(i)=2+i$. In the left figure, no coefficients are increased, which yields the limit~$\mathbb N$. In the right figure, all co\-ef\-fi\-cients are increased by one, which yields the limit~$\{p\in\mathbb Z\,|p\leq 0\}$.}
\label{fig:glueing-coeff-change}
\end{figure}

In Figure~\ref{fig:glueing-coeff-change} we see two ways to increase coefficients. Of course there are many more. In the direct limit, one can obtain any countable linear order (which just means that any such order is an increasing union of finite suborders). In particular, this includes all countable ordinals, which indicates the strength of our approach. Our next objective is a condition on the functions~$c_i$ which ensures that there is no descending sequence in the direct limit. Let us agree to abbreviate
\begin{equation*}
c_{ij}:=c_{j-1}\circ\ldots\circ c_i:\{0,\dots,b(i)-1\}\to\{0,\dots,b(j)-1\}
\end{equation*}
for $i<j$, where we assume that $c_k$ has (co-)domain as in the following.

\begin{definition}\label{def:Goodstein-system}
A Goodstein system is a pair $(b,c)$ of a non-decreasing function $b:\mathbb N\to\mathbb N$ with $b(0)>0$ and a family $c=\langle c_i\rangle_{i\in\mathbb N}$ of strictly increasing functions
\begin{equation*}
c_i:\{0,\dots,b(i)-1\}\to\{0,\dots,b(i+1)-1\}
\end{equation*}
with the following property: For any function $d:\mathbb N\to\mathbb N$ with $d(i)<b(i)$ for all~$i\in\mathbb N$ and any infinite $Y\subseteq\mathbb N$, there are $i,j\in Y$ with $i<j$ and $c_{ij}(d(i))\leq d(j)$.
\end{definition}

For $b(i)=2+i$ and $c_i(n)=n+1$ as in the example above, the constant function with values $d(i)=0$ witnesses that $(b,c)$ is no Goodstein system, as $i<j$ entails
\begin{equation*}
c_{ij}(d(i))=c_{ij}(0)=j-i>0=d(j).
\end{equation*}
This corresponds to the following observation about the right half of Figure~\ref{fig:glueing-coeff-change}: The points labelled zero in consecutive columns are mapped to a descending sequence in the direct limit. It is no coincidence that our counterexample is eliminated: We will see that all Goodstein sequences that are defined with respect to a Goodstein system do terminate. This fact will be called the extended Goodstein theorem.

The following result determines the precise strength of the extended Goodstein theorem for the binary representation (i.\,e.~for sequences $G_{b,c,m}^{2}(0),G_{b,c,m}^{2}(1),\dots$ as defined above). We recall that $\rca_0$ (recursive comprehension axiom) is a weak base theory that is commonly used in reverse mathematics (see~\cite{simpson09}). Arithmetical comprehension is a fundamental set existence principle, which is equivalent (over~$\rca_0$) to Ascoli's lemma and to K\H{o}nig's lemma for finitely branching trees. The extension of $\rca_0$ by arithmetical comprehension proves the same statements about natural numbers (i.\,e.~the same statements of first order arithmetic) as Peano arithmetic. Since $\rca_0$ and arithmetical comprehension are valid, the result does, in particular, prove the extended Goodstein theorem for the binary representation. We point out that the theorem will be restated and proved in the following section, which explains the numbering.

{
\def\thetheorem{\ref{thm:Goodstein-ACA}}
\addtocounter{theorem}{-1}
\begin{theorem}
The following are equivalent over~$\rca_0$:
\begin{enumerate}[label=(\roman*)]
\item arithmetical comprehension,
\item the extended Goodstein theorem for the binary representation: for any Goodstein system $(b,c)$ and any $m<2^{b(0)}$ there is an $i\in\mathbb N$ with $G_{b,c,m}^{2}(i)=0$.
\end{enumerate}
\end{theorem}
}

Let us observe that the definition of Goodstein systems involves universal quantification over infinite objects (``for all $d:\mathbb N\to\mathbb N$ and $Y\subseteq\mathbb N$"). In the extended Goodstein theorem the quantification becomes existential, since the notion of Goodstein system appears in an assumption. Hence the extended Goodstein theorem does entail the existence of infinite objects, as claimed above. This is unavoidable if we want an equivalence with arithmetical comprehension, which does itself assert the existence of uncomputable sets.

A detailed proof of Theorem~\ref{thm:Goodstein-ACA} will be given in the following section. Very briefly, the idea is to extend the operation $n\mapsto 2^n$ from natural numbers to (countable) linear orders: Its value on an order $(X,<_X)$ is defined as the set
\begin{equation*}
2^X=\{\langle x_0,\dots,x_{k-1}\rangle\,|\,x_0,\dots,x_{k-1}\in X\text{ with }x_{k-1}<_X\ldots <_X x_0\}
\end{equation*}
of finite descending sequences in~$X$, ordered lexicographically. It is known that arithmetical comprehension is equivalent to the statement that $X\mapsto 2^X$ preserves well orders (see~\cite{girard87,hirst94}). To prove that~(i) implies~(ii), one verifies that the direct limit~$X$ over a given Goodstein system~$(b,c)$ is well founded (cf.~Figure~\ref{fig:glueing-coeff-change}). One then shows that a Goodstein sequence $G_{b,c,m}^{2}(0),G_{b,c,m}^{2}(1),\dots$ can be mapped to a sequence in~$2^X$, which descends until the Goodstein sequence reaches zero. Since arithmetical comprehension ensures that~$2^X$ is well founded, this must happen eventually. To prove the converse, one observes that any countable well order arises as the direct limit over some Goodstein system $(b,c)$. One then shows that any descending sequence in~$2^X$ is majorized by the image of some Goodstein sequence (possibly after modifying~$b$). Using~(ii), one can conclude that $2^X$ contains no infinitely descending sequences. By the aforementioned result this suffices to establish~(i).

The argument from the previous paragraph remains valid if we replace the binary representation by other representations with suitable uniqueness and monotonicity properties. A general framework is provided by J.-Y.~Girard's notion of dilator~\cite{girard-pi2} (cf.~the related approaches by P.~Aczel~\cite{aczel-normal-functors} and S.~Feferman~\cite{feferman68}, as well as the discussion in~\cite{rathjen-well-ordering-principles}): Consider the category~$\lo$ of linear orders with the order embeddings (strictly increasing functions) as morphisms. The category~$\nat$ of natural numbers is the full subcategory with objects~$b:=\{0,\dots,b-1\}$ for $b\in\mathbb N$ (ordered as usual). Let us now introduce one of our central objects of study:

\begin{definition}\label{def:Goodstein-dilator}
A Goodstein dilator is a functor $D:\nat\to\lo$ such that
\begin{enumerate}[label=(\roman*)]
\item the orders $D(b)$ for $b\in\mathbb N$ are equal (not just isomorphic) to initial segments of~$\mathbb N$ (with the usual order between natural numbers),
\item the functor $D$ preserves pullbacks,
\item we have $0\in D(0)$ and $D(\emp_b)(0)=0\in D(b)$ for all numbers~$b\in\mathbb N$, where we write $\emp_b:0=\emptyset\to\{0,\dots,b-1\}=b$ for the empty function.
\end{enumerate}
\end{definition}

In view of condition~(i), we consider elements of~$D(b)$ as numbers (not terms) that have a base~$b$ representation. Note in particular that $\sigma-1\in D(b)$ is defined for~\mbox{$0\neq\sigma\in D(b)$}. To change coefficients according to a strictly increasing function $c:b\to b'$, one simply applies the morphism $D(c):D(b)\to D(b')$. Condition~(ii) ensures that these morphisms determine representations in a more familiar sense (cf.~the expression $D(e_a^n)(\sigma)$ in the proof of Proposition~\ref{prop:dil-extend} below, where $a$ can be seen as a list of coefficients). Also, condition~(ii) will allow us to define a well behaved extension $\overline D:\lo\to\lo$ of~$D$ into a transformation of linear orders (see~\cite{girard-pi2} and the following section). We remark that conditions~(i) and~(ii) are similar to the definition of weakly finite dilators, except that the latter must have finite values $D(b)$ for $b\in\mathbb N$. We allow $D(b)=\mathbb N$, in order to account for the possibility that every number has a base~$b$ representation. Let us point out that our definition does not require $\overline D$ to preserve well orders. This means that the extension of a Goodstein dilator may be a pre-dilator rather than a dilator in the strict sense. Condition~(iii) can be seen as a very weak form of well foundedness: It ensures that the elements $0\in D(b)$ do not map to a descending sequence in the direct limit~$\overline D(\mathbb N)$. We will use this condition in the proof of Theorem~\ref{thm:Goodstein-dilator} below. It can probably be weakened, but the present version of condition~(iii) is satisfied in our applications and appears quite natural: Intuitively speaking, it asserts that the number zero has a representation without coefficients that is independent of the base. We will study the following general notion of Goodstein sequence:

\begin{definition}\label{def:gen-Goodstein-sequence}
For a Goodstein dilator~$D$, a Goodstein system $(b,c)$ and a natural number~$m\in D(b(0))$, we define the Goodstein sequence $G_{b,c,m}^D(0),G_{b,c,m}^D(1),\dots$ by
\begin{equation*}
G_{b,c,m}^D(0)=m\quad\text{and}\quad G_{b,c,m}^D(i+1)=\begin{cases}
0 & \text{if $G_{b,c,m}^D(i)=0$},\\
D(c_i)(G_{b,c,m}^D(i))-1 & \text{otherwise}.
\end{cases}
\end{equation*}
When only given a non-decreasing $b:\mathbb N\to\mathbb N$ with $b(0)>0$ and an~$m\in D(b(0))$, we set $G^D_{b,m}(i):=G^D_{b,c,m}(i)$ for the functions $c_j:b(j)\to b(j+1)$ with $c_j(n)=n$.
\end{definition}

Note that we can indeed subtract one in the second case of the case distinction, since $0<G_{b,c,m}^D(i)$ in $D(b(i))$ entails $0\leq D(c_i)(0)<D(c_i)(G_{b,c,m}^D(i))$ in $D(b(i+1))$. By generalizing the argument for Theorem~\ref{thm:Goodstein-ACA} that we have sketched above, one can prove (a suitable formalization of) the next theorem. Conversely, the official proof that we give in the next section will deduce Theorem~\ref{thm:Goodstein-ACA} from the following.

{
\def\thetheorem{\ref{thm:Goodstein-dilator}}
\addtocounter{theorem}{-1}
\begin{theorem}
The theory $\rca_0$ proves that the following are equivalent for any Goodstein dilator~$D$:
\begin{enumerate}[label=(\roman*)]
\item the extended Goodstein theorem for~$D$: for any Goodstein system $(b,c)$ and any $m\in D(b(0))$ there is an $i\in\mathbb N$ with $G_{b,c,m}^D(i)=0$,
\item the extension $\overline D:\lo\to\lo$ of $D:\nat\to\lo$ preserves well foundedness, i.\,e.~the order $\overline D(X)$ is well founded for any well order~$X$.
\end{enumerate}
\end{theorem}
}

The equivalence in the theorem remains valid when one fixes $X=\mathbb N$ in~(ii) and, at the same time, disallows coefficient changes in~(i), by demanding $c_i(n)=n$ for all $n<b(i)$. For the Goodstein dilator~$D$ that is induced by the hereditary exponential notation, this equivalence (but not the one with coefficient changes) was proved in Goodstein's original paper~\cite{goodstein44}. A more detailed formulation of the general result without coefficient changes is provided by Theorem~\ref{thm:Goodstein-Pi11} below.

The idea to investigate Goodstein sequences in terms of dilators is due to M.~Abrusci, J.-Y.~Girard and J.~van~de Wiele~\cite{abrusci-goodstein,AGW-goodstein}. The latter have focused on conrete versions of Goodstein's theorem, such as the one of Kirby and Paris (``finite combinatorics"). They have not considered more abstract (``infinitary") results in reverse mathematics, such as our Theorems~\ref{thm:Goodstein-ACA} and~\ref{thm:Goodstein-dilator}. In their papers, Goodstein sequences are defined with respect to a given dilator, but coefficients are not allowed to increase.

The binary representation of natural numbers can be conceived as a Goodstein dilator with extension $X\mapsto 2^X$, as we shall see in the following section. In view of this fact, Theorem~\ref{thm:Goodstein-dilator} reduces Theorem~\ref{thm:Goodstein-ACA} to the known equivalence between arithmetical comprehension and the statement that $X\mapsto 2^X$ preserves well foundedness. The latter is an example of a well ordering principle. Many equivalences between well ordering principles and important set existence principles can be found in the literature (see~\cite{marcone-montalban,rathjen-afshari,rathjen-weiermann-atr,rathjen-atr,rathjen-model-bi,thomson-rathjen-Pi-1-1}). Do they all give rise to instances of the extended Goodstein theorem? Let us explore this question in a specific case:

The Veblen hierarchy of normal functions on the ordinals (see~\cite{schuette77,marcone-montalban}) can be cast as a transformation $X\mapsto\varphi_{1+X}0$ of linear orders (where $1+X$ represents the extension of~$X$ by a new minimum element). By a result of H.~Friedman (see~\cite{rathjen-weiermann-atr,marcone-montalban} for published proofs), the following are equivalent over $\rca_0$:
\begin{itemize}
\item arithmetical transfinite recursion,
\item the transformation $X\mapsto\varphi_{1+X}0$ preserves well foundedness.
\end{itemize}
We recall that arithmetical transfinite recursion is a much stronger set existence principle than arithmetical comprehension. It is equivalent (over $\rca_0$) to Ulm's theorem on Abelian groups and to the clopen Ramsey theorem (see~\cite{simpson09}).

The transformation $X\mapsto\varphi_{1+X}0$ can be turned into a dilator (cf.~\cite{weiermann-veblen-dilator}), but the latter does not arise as the extension of a Goodstein dilator: Not even the value $\varphi_10\cong\varepsilon_0$ for the finite argument $X=0=\emptyset$ is isomorphic to an initial segment of~$\mathbb N$. For this reason, Theorem~\ref{thm:Goodstein-dilator} does not apply immediately. One can try, however, to ``slow down" the transformation $X\mapsto\varphi_{1+X}0$, so that it preserves finite arguments but ``catches up" on infinite ones. It is known that this is possible in the case of ordinal exponentiation: The operation $\alpha\mapsto\omega^\alpha$ does not arise as the extension of a Goodstein dilator, since $\omega^2$ is no initial segment of~$\mathbb N$. However, the equation~$\omega^\alpha=2^{\omega\cdot\alpha}$ ensures that many important properties are inherited by the slower transformation $\alpha\mapsto 2^\alpha$. We have seen that the latter does arise from a Goodstein dilator, which allows us to apply Theorem~\ref{thm:Goodstein-dilator} to deduce Theorem~\ref{thm:Goodstein-ACA}.

As one of our main results, we will show that the Veblen hierarchy is closely connected to a Goodstein dilator that is based on a variant of the Ackermann function (cf.~\cite{ackermann28}). Consider the fast-growing hiearchy of functions $F_b:\mathbb N\to\mathbb N$ for~$b\in\mathbb N$, as determined by the recursive clauses
\begin{equation*}
F_0(n)=n+1\quad\text{and}\quad F_{b+1}(n)=F^{1+n}_b(n)=\underbrace{F_b\circ\ldots\circ F_b}_{1+n\text{ iterations}}(n).
\end{equation*}
Our variant of the Ackermann function is given by $b\mapsto F_b(1)$. In Section~\ref{sect:Ackermann-to-Veblen} we will define a Goodstein dilator $A:\nat\to\lo$ with values
\begin{equation*}
A(b)=\{0,\dots,F_b(1)-1\}
\end{equation*}
on objects $b\in\mathbb N$ (note that it remains to specify the action on morphisms). It is well-known that $\rca_0$ cannot prove the totality of the Ackermann function, since the latter grows faster than any primitive recursive function. For the following considerations, we thus extend our base theory by the principle~$\isigma_2$ of \mbox{$\Sigma^0_2$-induction} (see~the beginning of Section~\ref{sect:Ackermann-to-Veblen} for more details).

Once we have defined the action of $A:\nat\to\lo$ on morphisms (see Section~\ref{sect:Ackermann-to-Veblen}), we obtain an extension into a functor $\overline A:\lo\to\lo$. The following result relates this extension to the Veblen hierarchy. For orders~$X$ and $Y$, we write $X\preceq Y$ to indicate that there is an embedding of $X$ into~$Y$. The sum~$X+Y$ is defined as the disjoint union with $x<_{X+Y}x'<_{X+Y}y<_{X+Y}y'$ for elements $x<_X x'$ of $X$ and elements $y<_Y y'$ of $Y$. In the product $X\times Y$, an inequality $(x,y)<_{X\times Y}(x',y')$ holds if we have $x<_X x'$, or $x=x'$ and $y<_Y y'$. By $2+X=1+(1+X)$ we denote the extension of~$X$ by two bottom elements.

{
\def\thetheorem{\ref{thm:Ackermann-Veblen}}
\addtocounter{theorem}{-1}
\begin{theorem}[$\rca_0+\isigma_2$]
We have
\begin{equation*}
\overline A(X)\preceq\varphi_{1+X}0\preceq\overline A((2+X)\times\mathbb N)
\end{equation*}
for any linear order~$X$.
\end{theorem}
}

We point out that $X\cong\alpha$ entails $(2+X)\times\mathbb N\cong\omega\cdot(2+\alpha)$ in the sense of ordinal arithmetic. When $\alpha$ is a non-zero multiple of~$\omega^\omega$ (e.\,g.~an $\varepsilon$-number), then we have $\omega\cdot(2+\alpha)=\alpha=1+\alpha$, so that the orders in the previous theorem are isomorphic.

Definition~\ref{def:gen-Goodstein-sequence} yields Goodstein sequences $G_{b,c,m}^A(0),G_{b,c,m}^A(1),\dots$ for the Goodstein dilator~$A:\nat\to\lo$ (which also depend on the action of~$A$ on morphisms). By Theorems~\ref{thm:Goodstein-dilator} and~\ref{thm:Ackermann-Veblen}, the associated version of Goodstein's theorem is equivalent to the statement that $X\mapsto\varphi_{1+X}0$ preserves well foundedness. Together with Friedman's result from above, we obtain the following:

{
\def\thetheorem{\ref{thm:Goodstein-ATR}}
\addtocounter{theorem}{-1}
\begin{theorem}
The following are equivalent over~$\rca_0+\isigma_2$:
\begin{enumerate}[label=(\roman*)]
\item arithmetical transfinite recursion,
\item when $X$ is a well order, then so is $\overline A(X)$ (where $\overline A:\lo\to\lo$ extends the Goodstein dilator~$A:\nat\to\lo$ based on the Ackermann function),
\item the extended Goodstein theorem for the Goodstein dilator~$A$: for any Goodstein system $(b,c)$ and any $m\in A(b(0))$ there is an $i\in\mathbb N$ with $G_{b,c,m}^A(i)=0$.
\end{enumerate}
\end{theorem}
}

It is interesting to compare this theorem with a result of T.~Arai, D.~Fern\'andez-Duque, S.~Wainer and A.~Weiermann~\cite{AFWW-goodstein}, which shows that a different  Ackermannian Goodstein theorem is independent of~$\rca_0$ plus arithmetical transfinite recursion.

\section{Goodstein sequences and well ordering principles}

In this section we explain how a Goodstein dilator $D:\nat\to\lo$ can be extended into a transformation $\overline D:\lo\to\lo$ of linear orders. The construction is due to Girard~\cite{girard-pi2} and has also been presented in~\cite{freund-computable}. Since the setting in both cited sources is somewhat different from ours, we recall the most relevant arguments. Once the construction of~$\overline D$ is complete, we prove Theorem~\ref{thm:Goodstein-dilator} from the introduction: Goodstein sequences with respect to~$D$ terminate if, and only if, $\overline D$ preserves well foundedness. The binary representation of natural numbers will be used as a running example. The information from this example allows us to derive Theorem~\ref{thm:Goodstein-ACA} at the end of the section.

The notion of Goodstein dilator has been explained in Definition~\ref{def:Goodstein-dilator}. Working in second order arithmetic, we assume that Goodstein dilators are represented by subsets of~$\mathbb N$: Both objects $n:=\{0,\dots,n-1\}$ and morphisms $f:m\to n$ of the category $\nat$ can be coded by natural numbers. A functor $D:\nat\to\lo$ can thus be given as a set of tuples $(n,\sigma)$ with $\sigma\in D(n)$, $(n,\sigma,\tau)$ with $\sigma<_{D(n)}\tau$, and $(f,\sigma,\tau)$ with $D(f)(\sigma)=\tau$. 

To decide whether a given functor $D:\nat\to\lo$ is a Goodstein dilator, we need to check if it preserves pullbacks. Our next objective is a criterion that facilitates this task. As preparation, we introduce the finite powerset functor $[\cdot]^{<\omega}$ on the category of sets, with
\begin{align*}
[X]^{<\omega}&=\text{``the set of finite subsets of~$X$"},\\
[f]^{<\omega}(a)&=\{f(x)\,|\,x\in a\}\quad\text{(with $f:X\to Y$ and $a\in[X]^{<\omega}$)}.
\end{align*}
For $X=n=\{0,\dots,n-1\}$, the set $[n]^{<\omega}$ is, of course, the full power set. We will also write $[X]^{<\omega}$ when $X$ is a linear order, omitting the forgetful functor to its underlying set. Conversely, subsets of the latter will often be considered as suborders. We can now formulate the promised criterion. Note that $\rng(f)$ denotes the range (in the sense of image) of a function~$f$.

\begin{proposition}[$\rca_0$]\label{prop:pullback-nat-transf}
The following are equivalent for $D:\nat\to\lo$:
\begin{enumerate}[label=(\roman*)]
\item the functor~$D$ preserves pullbacks,
\item there is a natural transformation $\supp:D\Rightarrow[\cdot]^{<\omega}$ such that
\begin{equation*}
\supp_n(\sigma)\subseteq\rng(f)\quad\Rightarrow\quad\sigma\in\rng(D(f))
\end{equation*}
holds for any morphism $f:m\to n$ and any $\sigma\in D(n)$.
\end{enumerate}
If a natural transformation as in~(ii) exists, then it is unique.
\end{proposition}

The proposition is essentially implicit in~\cite[Theorem~2.3.12]{girard-pi2}. Since our setting is somewhat different, we have decided to provide a proof nevertheless, which the reader can find in Appendix~\ref{appendix:pullback-support} to the present paper. Using Proposition~\ref{prop:pullback-nat-transf}, it is straightforward to verify that the binary representation of numbers gives rise to a Goodstein dilator, as promised in the introduction:

\begin{example}\label{ex:2-Goodstein}
Any natural number below~$2^n$ can be uniquely written in the form $2^{n_0}+\dots+2^{n_{k-1}}$ with $n>n_0>\ldots>n_{k-1}$ (take $k=0$ to represent zero). In the present example we always assume that numbers are written in this form, i.\,e.~with descending exponents. The lexicographic order between the lists of exponents coincides with the usual order between the represented numbers. We can thus define a functor $2:\nat\to\lo$ by setting
\begin{align*}
2(n)&:=\{0,\dots,2^n-1\},\\
2(f)(2^{n_0}+\dots+2^{n_{k-1}})&:=2^{f(n_0)}+\dots+2^{f(n_{k-1})}\quad\text{(for a morphism $f:n\to m$)}.
\end{align*}
In order to show that we have a Goodstein dilator, we must prove that our functor preserves pullbacks. To apply the criterion from Proposition~\ref{prop:pullback-nat-transf}, we define a family of functions $\supp_n:2(n)\to[n]^{<\omega}$ by setting
\begin{equation*}
\supp_n(2^{n_0}+\dots+2^{n_{k-1}}):=\{n_0,\ldots,n_{k-1}\}.
\end{equation*}
Naturality is readily verified. Now assume we have $\supp_n(2^{n_0}+\dots+2^{n_{k-1}})\subseteq\rng(f)$ for a morphism $f:m\to n$. For $i<k$ we can define $m_i$ by stipulating $f(m_i)=n_i$. Since the morphism $f$ is order preserving, we have $m>m_0>\ldots>m_{k-1}$. This means that $2^{m_0}+\dots+2^{m_{k-1}}\in 2(m)$ is represented in the required form. By construction, we get
\begin{equation*}
2^{n_0}+\dots+2^{n_{k-1}}=2(f)(2^{m_0}+\dots+2^{m_{k-1}})\in\rng(2(f)),
\end{equation*}
as required for Proposition~\ref{prop:pullback-nat-transf}. Let us also observe that we have $0\in\{0\}=2(0)$ and $2(f)(0)=0$ for any morphism~$f$, so that condition~(iii) from Definition~\ref{def:Goodstein-dilator} is satisfied. In the introduction we have, on the one hand, given an ad hoc definition of Goodstein sequences $G_{b,c,m}^{2}(0),G_{b,c,m}^{2}(1),\dots$ for the binary representation. On the other hand, Definition~\ref{def:gen-Goodstein-sequence} yields Goodstein sequences $G_{b,c,m}^{D}(0),G_{b,c,m}^{D}(1),\dots$ relative to an arbitrary Goodstein dilator~$D$. One readly checks that the two definitions coincide when $D$ is the Goodstein dilator~$2$ that we have just defined.
\end{example}

Our next objective is the extension of a Goodstein dilator~$D:\nat\to\lo$ into a functor $\overline D:\lo\to\lo$. The crucial idea (due to~\cite{girard-pi2}) is to view a given linear order~$X$ as the direct limit of its finite suborders~$a\in[X]^{<\omega}$. Let us write $|a|\in\mathbb N$ for the cardinality of such an order. Intuitively speaking, the unique isomorphism $a\cong |a|=\{0,\dots,|a|-1\}$ (where the right side is ordered as usual) should induce an isomorphism $\overline D(a)\cong D(|a|)$. Note that $\overline D(a)$ is not officially explained yet, while $D(|a|)$ is defined when $D$ is a Goodstein dilator. Since $\overline D$ is to be a functor, the inclusion $a\hookrightarrow X$ should give rise to a morphism $D(|a|)\cong\overline D(a)\to\overline D(X)$. Now the idea is to define $\overline D(X)$ as the direct limit over these morphisms. Concretely, we will define the underlying set of $\overline D(X)$ as a collection of pairs~$(a,\sigma)$ with $a\in[X]^{<\omega}$ and $\sigma\in D(|a|)$. Intuitively speaking, such a pair represents the image of $\sigma$ under the indicated morphism.

The set $\overline D(X)$ should not, however, contain all pairs of the indicated form. To explain why, we introduce some notation that will also be needed below: Given an embedding $f:a\to b$ between finite orders, let us write $|f|:|a|\to|b|$ for the unique morphism that makes the following diagram commute:
\begin{equation*}
\begin{tikzcd}[row sep=large]
a\ar[d,"\cong"]\ar[r,"f"] & b\ar[d,"\cong"]\\
{|a|=\{0,\dots,|a|-1\}}\ar[r,"{|f|}"] & {|b|=\{0,\dots,|b|-1\}}
\end{tikzcd}
\end{equation*}
We also agree to write $\iota_a^b:a\hookrightarrow b$ for the inclusion between suborders $a\subseteq b\in[X]^{<\omega}$. Intuively speaking, the following diagram should commute:
\begin{equation*}
\begin{tikzcd}[row sep=tiny]
D({|a|})\cong\overline D(a)\ar[rd]\ar[dd,"{D(|\iota_a^b|)}",shift right=1.5em,swap] & \\
& \overline D(X)\\
D({|b|})\cong\overline D(b)\ar[ru] &
\end{tikzcd}
\end{equation*}
In terms of our representations, this means that the pairs $(a,\sigma)$ and $(b,D(|\iota_a^b|)(\sigma))$ should correspond to the same element of~$\overline D(X)$. In search for a criterion that excludes the ``superfluous" representation $(b,D(|\iota_a^b|)(\sigma))$, we observe the following: Let $\supp:D\Rightarrow[\cdot]^{<\omega}$ be the natural transformation provided by Proposition~\ref{prop:pullback-nat-transf}. If we have $a\subsetneq b$, then $|\iota_a^b|$ is not surjective, so that we get
\begin{equation*}
\supp_{|b|}(D(|\iota_a^b|)(\sigma))=[|\iota_a^b|]^{<\omega}\circ\supp_{|a|}(\sigma)\subseteq\rng(|\iota_a^b|)\subsetneq |b|=\{0,\dots,|b|-1\}.
\end{equation*}
On an intuitive level, the fact that we have a proper inclusion means that~$b$ is bigger than required. We will see that the opposed condition $\supp_{|a|}(\sigma)=|a|$ in the following definition suffices to guarantee unique representations. The following coincides with~\cite[Definition~2.2]{freund-computable}.

\begin{definition}[$\rca_0$]\label{def:dil-extend}
Consider a Goodstein dilator~$D$. For each linear order~$X$, we define a set $\overline D(X)$ by
\begin{equation*}
\overline D(X):=\{(a,\sigma)\,|\,a\in[X]^{<\omega}\text{ and }\sigma\in D(|a|)\text{ with }\supp_{|a|}(\sigma)=|a|\},
\end{equation*}
where $\supp:D\Rightarrow[\cdot]^{<\omega}$ is the natural transformation from Proposition~\ref{prop:pullback-nat-transf}. To define a binary relation $<_{\overline D(X)}$ on the set $\overline D(X)$, we stipulate
\begin{equation*}
(a,\sigma)<_{\overline D(X)}(b,\tau)\quad:\Leftrightarrow\quad D(|\iota_a^{a\cup b}|)(\sigma)<_{D(|a\cup b|)}D(|\iota_b^{a\cup b}|)(\tau).
\end{equation*}
Given an order embedding~$f:X\to Y$, we define a map $\overline D(f):\overline D(X)\to\overline D(Y)$ by
\begin{equation*}
\overline D(f)((a,\sigma)):=([f]^{<\omega}(a),\sigma),
\end{equation*}
which is permitted in view of $|[f]^{<\omega}(a)|=|a|$.
\end{definition}

The following result does, in particular, assert that $(\overline D(X),<_{\overline D(X)})$ is a linear order whenever the same holds for $X=(X,<_X)$.

\begin{proposition}[$\rca_0$]
If~$D$ is a Goodstein dilator, then the constructions from Definition~\ref{def:dil-extend} yield a functor~$\overline D:\lo\to\lo$.
\end{proposition}
\begin{proof}
In~\cite[Lemma~2.2]{freund-computable}, the same result has been shown in a stronger base theory. It is straightforward to check that the argument goes through in~$\rca_0$ as well. We only recall one particularly instructive part of the argument: The claim that $<_{\overline D(X)}$ is trichotomous is readily reduced to the implication
\begin{equation*}
D(|\iota_a^{a\cup b}|)(\sigma)=D(|\iota_b^{a\cup b}|)(\tau)\quad\Rightarrow\quad (a,\sigma)=(b,\tau).
\end{equation*}
Assuming that the antecendent holds, we can use the condition $\supp_{|a|}(\sigma)=|a|$ from Definition~\ref{def:dil-extend} and naturality to deduce
\begin{multline*}
\rng(|\iota_a^{a\cup b}|)=[|\iota_a^{a\cup b}|]^{<\omega}\circ\supp_{|a|}(\sigma)=\supp_{|a\cup b|}\circ D(|\iota_a^{a\cup b}|)(\sigma)=\\
=\supp_{|a\cup b|}\circ D(|\iota_b^{a\cup b}|)(\tau)=[|\iota_b^{a\cup b}|]^{<\omega}\circ\supp_{|b|}(\tau)=\rng(|\iota_b^{a\cup b}|).
\end{multline*}
By the definition of $|\cdot|$, the diagram
\begin{equation*}
\begin{tikzcd}[row sep=large,column sep=large]
a\ar[d,"\cong",swap]\ar[r,"\iota_a^{a\cup b}"] & a\cup b\ar[d,"\cong"] & b\ar[d,"\cong"]\ar[l,"\iota_b^{a\cup b}",swap]\\
{|a|}\ar[r,"{|\iota_a^{a\cup b}|}",swap] & {|a\cup b|} & {|b|}\ar[l,"{|\iota_b^{a\cup b}|}"]
\end{tikzcd}
\end{equation*}
commutes. One readily infers $|a\cup b|=\rng(|\iota_a^{a\cup b}|)\cup\rng(|\iota_b^{a\cup b}|)$. Together with the above we get $\rng(|\iota_a^{a\cup b}|)=|a\cup b|=\rng(|\iota_b^{a\cup b}|)$, which entails $|a|=|a\cup b|=|b|$ and hence $a=b$. It follows that $|\iota_a^{a\cup b}|=|\iota_b^{a\cup b}|$ and hence $D(|\iota_a^{a\cup b}|)=D(|\iota_b^{a\cup b}|)$ are the identity (on $|a\cup b|$ and $D(|a\cup b|)$, respectively). Thus the antecedent of the above implication also yields $\sigma=\tau$, as required.
\end{proof}

Extending the constructions from Definition~\ref{def:dil-extend}, one can define a natural transformation $\supp:\overline D\to[\cdot]^{<\omega}$ by setting $\overline\supp_X((a,\sigma)):=a$. It is straightforward to verify an implication as in part~(ii) of Proposition~\ref{prop:pullback-nat-transf}. One can conclude that the functor~$\overline D:\lo\to\lo$ preserves direct limits as well as pullbacks. If it also preserves well foundedness, then it is a dilator in the usual sense (cf.~\cite{girard-pi2}). These additional observations are not strictly needed in the present paper. For details we refer to~\cite[Lemma~2.2]{freund-computable}. Let us now show that $\overline D$ is indeed an extension of~$D$.

\begin{proposition}[$\rca_0$]\label{prop:dil-extend}
For each Goodstein dilator~$D:\nat\to\lo$ there is a natural isomorphism~$\eta^D:\overline D\!\restriction\!\nat\Rightarrow D$, where $\overline D\!\restriction\!\nat$ denotes the restriction of the functor $\overline D:\lo\to\lo$ to the category of natural numbers.
\end{proposition}
\begin{proof}
This has been shown in~\cite[Lemma~2.6]{freund-rathjen_derivatives}. We reproduce the most informative part of the argument: Let us recall that elements of $\overline D(n)$ have the form $(a,\sigma)$ with $a\subseteq n=\{0,\dots,n-1\}$ and $\sigma\in D(|a|)$. Writing $e_a^n:|a|\to n$ for the strictly increasing function with range~$a$, the component $\eta^D_n:\overline D(n)\to D(n)$ can be given~by
\begin{equation*}
\eta^D_n((a,\sigma)):=D(e_a^n)(\sigma).
\end{equation*}
In order to show that $\eta^D_n$ is surjective, we consider an arbitrary element~$\tau\in D(n)$. Let us set $a:=\supp_n(\tau)$, where $\supp:D\Rightarrow[\cdot]^{<\omega}$ is the natural transformation from Proposition~\ref{prop:pullback-nat-transf}. According to the latter, we get $\tau=D(e_a^n)(\sigma)$ for some $\sigma\in D(|a|)$. Due to naturality, we see
\begin{equation*}
[e_a^n]^{<\omega}(\supp_{|a|}(\sigma))=\supp_n(D(e_a^n)(\sigma))=\supp_n(\tau)=a,
\end{equation*}
which forces $\supp_{|a|}(\sigma)=a$. In view of Definition~\ref{def:dil-extend}, this yields $(a,\sigma)\in\overline D(n)$. By construction we have
\begin{equation*}
\tau=D(e_a^n)(\sigma)=\eta^D_n((a,\sigma))\in\rng(\eta^D_n),
\end{equation*}
as desired. To see that $(a,\sigma)<_{\overline D(n)}(b,\tau)$ implies $\eta^D_n(\sigma)<_{D(n)}\eta^D_n(\tau)$, one applies Definition~\ref{def:dil-extend} and observes that the following diagram commutes (which is the case since all arrows are of the form $|\iota|$ for an inclusion~$\iota$):
\begin{equation*}
\begin{tikzcd}[row sep=large]
{|a|}\ar[r,"{|\iota_a^{a\cup b}|}"]\ar[rd,"e_a^n",swap] & {|a\cup b|}\ar[d,"e_{a\cup b}^n"] & {|b|}\ar[l,"{|\iota_a^{a\cup b}|}",swap]\ar[ld,"e_b^n"]\\
& n={|n|} &
\end{tikzcd}
\end{equation*}
Naturality with respect to a morphism $f:m\to n$ is readily deduced from the fact that we have $f\circ e_a=e_{[f]^{<\omega}(a)}$ for $a\subseteq m$ (note that both functions enumerate the set $[f]^{<\omega}(a)\subseteq n$ in increasing order).
\end{proof}

It will also be convenient to know that the previous proposition determines $\overline D$ in the following sense:

\begin{proposition}[$\rca_0$]\label{prop:determine-extension}
Consider a Goodstein dilator~$D:\nat\to\lo$ and a functor $\widehat D:\lo\to\lo$. Assume that
\begin{enumerate}[label=(\roman*)]
\item there is a natural isomorphism between $D$ and the restriction of $\widehat D$ to the category of natural numbers, and
\item there is a natural transformation $\widehat\supp:\widehat D\Rightarrow[\cdot]^{<\omega}$ such that
\begin{equation*}
\widehat\supp_Y(\sigma)\subseteq\rng(f)\quad\Rightarrow\quad\sigma\in\rng(\widehat D(f))
\end{equation*}
holds for any order embedding $f:X\to Y$ and any element $\sigma\in D(Y)$.
\end{enumerate}
Then there is a natural isomorphism between~$\widehat D$ and the extension~$\overline D$ of~$D$.
\end{proposition}
To be precise, we should point out that the functor $\widehat D:\lo\to\lo$ cannot be given as a subset of~$\mathbb N$ (not even if we restrict to countable orders). We assume that $\widehat D$ is given via $\Delta^0_1$-definitions of the relations $\sigma\in\widehat D(X)$, $\sigma<_{\widehat D(X)}\tau$ and $\widehat D(f)(\sigma)=\tau$. The following argument transforms these into a $\Delta^0_1$-definition of~$\eta_X(\sigma)=\tau$, where we write $\eta_X:\overline D(X)\to\widehat D(X)$ for the components of the desired isomorphism.
\begin{proof}
Assumption~(ii) ensures that $\widehat D$ is a prae-dilator in the sense of~\cite[Section~2]{freund-computable} (and a dilator if it preserves well foundedness). The result can now be obtained by combining Proposition~2.1 and Lemma~2.3 of~\cite{freund-computable}. To describe the desired isomorphism $\eta:\overline D\Rightarrow\widehat D$ more explicitly, we write $\eta^0:D\to\widehat D\!\restriction\!\nat$ for the natural isomorphism provided by assumption~(i). For a linear order~$X$, we define $e_a^X:|a|\to X$ as the embedding with range~$a\in[X]^{<\omega}$. We can now set
\begin{equation*}
\eta_X((a,\sigma)):=\widehat D(e_a^X)\circ\eta^0_{|a|}(\sigma).
\end{equation*}
To verify the required properties, one argues as in the proof of Proposition~\ref{prop:dil-extend}. Where the latter refers to Proposition~\ref{prop:pullback-nat-transf}, one now uses assumption~(ii) of the present proposition.
\end{proof}

Our next goal is to describe the extension $\overline 2:\lo\to\lo$ of the Goodstein dilator~$2:\nat\to\lo$ from Example~\ref{ex:2-Goodstein}. The condition~$\supp_{|a|}(\sigma)=|a|$ from Definition~\ref{def:dil-extend} amounts to
\begin{equation*}
\{n_0,\dots,n_{k-1}\}=\supp_n(2^{n_0}+\dots+2^{n_{k-1}})=n=\{0,\dots,n-1\},
\end{equation*}
which requires $n=k$ and $n_i=k-1-i$ (note $n_0>\dots>n_{k-1}$). This means that elements of $\overline 2(X)$ are of the form $(a,2^{k-1}+\dots+2^0)=(a,2^k-1)$, where $a\in[X]^{<\omega}$ has $k$ elements. For $a=\{x_0,\dots,x_{k-1}\}$ with $x_{k-1}<_X\dots<_X x_0\in X$, the pair $(a,2^k-1)$ does intuitively correspond to the element $\langle x_0,\dots,x_{k-1}\rangle$ of the order~$2^X$ from the introduction. To justify this claim in detail, we would need to analyse the order relation~$<_{\overline 2(X)}$ that is determined by Definition~\ref{def:dil-extend}. This is tedious, even in the relatively simple case at hand. Fortunately, Proposition~\ref{prop:determine-extension} suggests an alternative approach:

\begin{example}\label{ex:2-Goodstein-extend}
Given a linear order~$X=(X,<_X)$, we consider the set
\begin{equation*}
\widehat 2(X):=\{\langle x_0,\dots,x_{k-1}\rangle\,|\,x_0,\dots,x_{k-1}\in X\text{ with }x_{k-1}<_X\ldots <_X x_0\}
\end{equation*}
with the lexicographic order, which is given by
\begin{equation*}
\langle x_0,\dots,x_{k-1}\rangle<_{\widehat 2(X)}\langle x'_0,\dots,x'_{m-1}\rangle\quad\Leftrightarrow\quad
\begin{cases}
\text{either $x_i=x'_i$ for all $i<k<m$,}\\[.5ex]
\parbox[t]{.41\textwidth}{or there is a $j<\min\{k,m\}$ with $x_j<_Xx'_j$ and $x_i=x'_i$ for all $i<j$.}
\end{cases}
\end{equation*}
In the introduction we have denoted the same order by~$2^X$, which is the most common notation in the literature. By changing the notation to $\widehat 2(X)$, we aim to distinguish the present construction from the one in Example~\ref{ex:2-Goodstein}. Given an order embedding $f:X\to Y$, we can define an embedding $\widehat 2(f):\widehat 2(X)\to\widehat 2(Y)$ by setting
\begin{equation*}
\widehat 2(f)(\langle x_0,\dots,x_{k-1}\rangle):=\langle f(x_0),\dots,f(x_{k-1})\rangle.
\end{equation*}
It is straightforward to see that this turns $\widehat 2:\lo\to\lo$ into a functor. The restriction of this functor to the category of natural numbers is isomorphic to the functor $2:\nat\to\lo$ from Example~\ref{ex:2-Goodstein}, as witnessed by the maps
\begin{equation*}
2(n)\ni 2^{n_0}+\dots+2^{n_{k-1}}\mapsto\langle n_0,\dots,n_{k-1}\rangle\in\widehat 2(n).
\end{equation*}
Let us point out that the argument $2^{n_0}+\dots+2^{n_{k-1}}$ is a natural number (rather than a term), which is represented according to the convention $n_0>\dots>n_{k-1}$ that we have fixed in Example~\ref{ex:2-Goodstein}. Now define a family of functions $\widehat\supp_X:\widehat 2(X)\to[X]^{<\omega}$ by setting
\begin{equation*}
\widehat\supp_X(\langle x_0,\dots,x_{k-1}\rangle):=\{x_0,\dots,x_{k-1}\}.
\end{equation*}
It is straightforward to see that the condition from Proposition~\ref{prop:determine-extension} is satisfied. We can thus conclude that there is a natural isomorphism between the functor $\widehat 2:\lo\to\lo$ and the extension $\overline 2:\lo\to\lo$ of the Goodstein dilator $2:\nat\to\lo$ from Example~\ref{ex:2-Goodstein}. In particular, the map $X\mapsto\widehat 2(X)$ preserves well foundedness (i.\,e.~is a dilator) if, and only if, the map $X\mapsto\overline 2(X)$ does, provably in $\rca_0$.
\end{example}

Let us now come to the main result of this section, which was already stated in the introduction. We refer to Definitions~\ref{def:Goodstein-system} and~\ref{def:gen-Goodstein-sequence} for the notion of Goodstein system and the general Goodstein sequences $G_{b,c,m}^D(0),G_{b,c,m}^D(1),\dots$, respectively.

\begin{theorem}\label{thm:Goodstein-dilator}
The theory $\rca_0$ proves that the following are equivalent for any Goodstein dilator~$D$:
\begin{enumerate}[label=(\roman*)]
\item the extended Goodstein theorem for~$D$: for any Goodstein system $(b,c)$ and any $m\in D(b(0))$ there is an $i\in\mathbb N$ with $G_{b,c,m}^D(i)=0$,
\item the extension $\overline D:\lo\to\lo$ of $D:\nat\to\lo$ preserves well foundedness, i.\,e.~the order $\overline D(X)$ is well founded for any well order~$X$.
\end{enumerate}
\end{theorem}
\begin{proof}
Let us first assume~(i) and deduce~(ii). Aiming at the latter, we consider a well order~$X$. We may assume that $X$ is non-empty, say~$\star\in X$: Otherwise, replace~$X$ by $X':=X\cup\{\star\}$; the obvious embedding $\iota:X\to X'$ induces an embedding $\overline D(\iota):\overline D(X)\to\overline D(X')$, so that $\overline D(X)$ is well founded if the same holds for $\overline D(X')$. Let us now consider an infinite sequence $(a_0,\sigma_0),(a_1,\sigma_1),\ldots$ in~$\overline D(X)$ (recall $a_i\in[X]^{<\omega}$ from Definition~\ref{def:dil-extend}). For $i\in\mathbb N$ we set
\begin{equation*}
b(i):=|\{\star\}\cup a_0\cup\dots\cup a_i|.
\end{equation*}
Then $b:\mathbb N\to\mathbb N$ is non-decreasing, and the presence of~$\star$ ensures~$b(0)>0$, as required by Definition~\ref{def:Goodstein-system}. For $i\leq j$ we define $c_{ij}:b(i)\to b(j)$ as the unique function that makes the following diagram commute:
\begin{equation*}
\begin{tikzcd}[row sep=small]
b(i)=\{0,\dots,b(i)-1\}\ar[dd,swap,"c_{ij}"]\ar[r,"\cong"] & \{\star\}\cup a_0\cup\dots\cup a_i\ar[dd,hook]\ar[rd,hook] & \\
& & X \\
b(j)=\{0,\dots,b(j)-1\}\ar[r,"\cong"] & \{\star\}\cup a_0\cup\dots\cup a_j\ar[ru,hook] &
\end{tikzcd}
\end{equation*}
Here the horizontal arrows are order isomorphisms with respect to the usual order on the sets $b(k)\subseteq\mathbb N$. This entails that $c_{ij}$ is strictly increasing. We set $c_i:=c_{i,i+1}$ to get $c_{ij}=c_{j-1}\circ\dots\circ c_i$, as in the paragraph before Definition~\ref{def:Goodstein-system}. To see that $b$ and $c:=\langle c_i\rangle_{i\in\mathbb N}$ form a Goodstein system in the sense of that definition, we consider a function $d:\mathbb N\to\mathbb N$ with $d(i)<b(i)$, as well as an infinite set~$Y\subseteq\mathbb N$. Let us enumerate the latter as $Y=\{y(0),y(1),\dots\}$ with $y(0)<y(1)<\dots$ in increasing order. We now define $d'(k)\in X$ as the image of $d(y(k))\in b(y(k))$ under the following function (cf.~the previous diagram):
\begin{equation*}
\begin{tikzcd}[row sep=small]
b(y(k))\ar[r,"\cong"] & \{\star\}\cup a_0\cup\dots\cup a_{y(k)}\ar[r,hook] & X\\
d(y(k))\ar[u,symbol=\in]\ar[rr,maps to] & & d'(k)\ar[u,symbol=\in]
\end{tikzcd}
\end{equation*}
Due to the assumption that $X$ is a well order, we must have $d'(k)\leq_X d'(k+1)$ for some~$k\in\mathbb N$. Set $i:=y(k)$ and $j:=y(k+1)$, and observe $i<j\in Y$. Combining the definition of $d'(k)$ with the diagram that defines $c_{ij}$, we see that $c_{ij}(d(i))\in b(j)$ is mapped to $d'(k)\in X$, while $d(j)\in b(j)$ is mapped to $d'(k+1)\in X$. Since the relevant function from~$b(j)$ to~$X$ is an embedding, we can invoke $d'(k)\leq_X d'(k+1)$ to infer $c_{ij}(d(i))\leq d(j)$. This completes the proof that $(b,c)$ is a Goodstein system in the sense of Definition~\ref{def:Goodstein-system}. Now consider the inclusions $\iota_i:a_i\to\{\star\}\cup a_0\cup\dots\cup a_i$ for $i\in\mathbb N$. According to a previous definition, the function $|\iota_i|:|a_i|\to b(i)$ is the unique morphism that makes the following diagram commute:
\begin{equation*}
\begin{tikzcd}[row sep=large]
a_i\ar[d,"\cong"]\ar[r,"\iota_i"] & \{\star\}\cup a_0\cup\dots\cup a_i\ar[d,"\cong"]\\
{|a_i|=\{0,\dots,|a_i|-1\}}\ar[r,"{|\iota_i|}"] & {b(i)=\{0,\dots,b(i)-1\}}
\end{tikzcd}
\end{equation*}
Let us now recall that the sets $a_i$ are the first components of pairs $(a_i,\sigma_i)\in\overline D(X)$. In view of Definition~\ref{def:dil-extend} we have $\sigma_i\in D(|a_i|)$, so that we can set
\begin{equation*}
m:=D(|\iota_0|)(\sigma_0)\in D(b(0)).
\end{equation*}
Definition~\ref{def:gen-Goodstein-sequence} provides a Goodstein sequence $G^D_{b,c,m}(0),G^D_{b,c,m}(1),\dots$ for the given Goodstein dilator~$D$ and $b,c,m$ as just defined.  The following will be crucial:
\begin{claim}
If we have $(a_0,\sigma_0)>_{\overline D(X)}\dots>_{\overline D(X)}(a_i,\sigma_i)$ and $G^D_{b,c,m}(j)\neq 0$ for all $j<i$, then we have $D(|\iota_i|)(\sigma_i)\leq_{D(b(i))} G^D_{b,c,m}(i)$.
\end{claim}
\noindent As preparation, we observe that
\begin{equation*}
(a_j,\sigma_j)<_{\overline D(X)}(a_k,\sigma_k)\quad\Leftrightarrow\quad D(c_{jl}\circ|\iota_j|)(\sigma_j)<_{D(b(l))} D(c_{kl}\circ|\iota_k|)(\sigma_k)
\end{equation*}
holds for $j,k\leq l$. This equivalence is readily reduced to Definition~\ref{def:dil-extend}, using the fact that the following diagram commutes, where each arrow between objects $|a|$ and~$|a'|$ represents the morphism $|\iota|$ that is induced by the inclusion~$\iota:a\hookrightarrow a'$.
\begin{equation*}
\begin{tikzcd}[row sep=large]
{|a_j|}\ar[r]\ar[d,"{|\iota_j|}",swap] & {|a_j\cup a_k|}\ar[d] & {|a_k|}\ar[l]\ar[d,"{|\iota_j|}"]\\
b(j)\ar[r,"c_{jl}",swap] & {b(l)=|\{\star\}\cup a_0\cup\dots\cup a_l|} & b(k)\ar[l,"c_{kl}"]
\end{tikzcd}
\end{equation*}
To prove our claim, we first note that $D(|\iota_0|)(\sigma_0)=m=G^D_{b,c,m}(0)$ holds by construction. Arguing by induction, we now assume $D(|\iota_i|)(\sigma_i)\leq_{D(b(i))} G^D_{b,c,m}(i)\neq 0$ and $(a_i,\sigma_i)>_{\overline D(X)}(a_{i+1},\sigma_{i+1})$. Using the equivalence that we have established as preparation (with $k=i$ and $j=i+1=l$, so that $c_{jl}$ is the identity), we obtain
\begin{multline*}
D(|\iota_{i+1}|)(\sigma_{i+1})\leq_{D(b(i+1))} D(c_i\circ|\iota_i|)(\sigma_i)-1\leq_{D(b(i+1))}\\
\leq_{D(b(i+1))} D(c_i)(G^D_{b,c,m}(i))-1=G^D_{b,c,m}(i+1),
\end{multline*}
which completes the proof of the induction step and hence of the claim.

Statement~(i) from the present theorem provides an $i\in\mathbb N$ with $G^D_{b,c,m}(i)=0$. We may assume that~$i$ is minimal with this property. Using the claim above, we now deduce that $(a_j,\sigma_j)\leq_{\overline D(X)}(a_{j+1},\sigma_{j+1})$ holds for some $j\leq i$, as required to show that $\overline D(X)$ is well founded. If the desired inequality does not hold for any~$j<i$, then the claim yields $D(|\iota_i|)(\sigma_i)=0$. Due to condition~(iii) of Definition~\ref{def:Goodstein-dilator} we also have $D(\emp_{b(i)})(0)=0$, where $\emp_{b(i)}:0\to b(i)$ is the empty function. In view of $c_i\circ\emp_{b(i)}=\emp_{b(i+1)}$ (both are the empty function) we get
\begin{multline*}
D(c_i\circ|\iota_i|)(\sigma_i)=D(c_i\circ\emp_{b(i)})(0)=\\
=D(\emp_{b(i+1)})(0)=0\leq_{D(b(i+1))} D(|\iota_{i+1}|)(\sigma_{i+1}).
\end{multline*}
The equivalence that we have shown in the proof of the claim remains valid when we replace both strict inequalities by non-strict ones, due to trichotomy. Thus the inequality that we have just established entails~$(a_i,\sigma_i)\leq_{\overline D(X)}(a_{i+1},\sigma_{i+1})$, as needed. This completes the proof that~(i) implies~(ii).

We now assume~(ii) and deduce~(i). Aiming at the latter, we consider a Goodstein system~$(b,c)$. Recall that we have $c=\langle c_i\rangle_{i\in\mathbb N}$ with functions~$c_i:b(i)\to b(i+1)$. As before, we write $c_{ij}:=c_{j-1}\circ\dots\circ c_i:b(i)\to b(j)$ for $i\leq j$. We consider the morphisms~$c_{ij}$ as a directed system in the category of linear orders. The direct limit over this system consists of embeddings $e_i:b(i)\to X$ into a linear order~$X$ with underlying set
\begin{equation*}
X=\bigcup_{i\in\mathbb N}\rng(e_i),
\end{equation*}
such that the following diagram commutes for any $i<j$:
\begin{equation*}
\begin{tikzcd}[row sep=small]
b(i)=\{0,\dots,b(i)-1\}\ar[rd,"e_i"]\ar[dd,"c_{ij}",swap] & \\
 & X \\
b(j)=\{0,\dots,b(j)-1\}\ar[ru,"e_j"]
\end{tikzcd}
\end{equation*}
Working in~$\rca_0$, one can explicitly construct $X$ as an order with underlying set
\begin{equation*}
\{(0,n)\,|\,n\in b(0)\}\cup\{(i+1,n)\,|\,i\in\mathbb N\text{ and }n\in b(i+1)\backslash\rng(c_i)\}.
\end{equation*}
The order between $(i,n)$ and $(j,n')$ is determined by comparing $c_{ik}(n)$ and $c_{jk}(n')$ for some number~$k\geq i,j$ (the precise value of which is irrelevant). A family of functions $e_j:b(j)\to X$ as above can be defined by stipulating that $e_j(n)=(i,n_0)$ holds for $c_{ij}(n_0)=n$ with $i\leq j$ as small as possible. Having constructed the direct limit~$X$, we now show that it is a well order. Given an arbitrary function~$f:\mathbb N\to X$, we construct a set $Y=\{y(0),y(1),\dots\}\subseteq\mathbb N$ so that we have $y(0)<y(1)<\dots$ and $f(i)\in\rng(e_{y(i)})$ for all~$i\in\mathbb N$ (which is possible since the union $X=\bigcup_{i\in\mathbb N}\rng(e_i)$ is increasing). We now consider the function $d:\mathbb N\to\mathbb N$ that is given by
\begin{equation*}
d(i):=\begin{cases}
n & \text{if $i=y(k)$ and $e_i(n)=f(k)$},\\
0 & \text{if $i\notin Y$}.
\end{cases}
\end{equation*}
Let us note that $d(i)<b(i)$ holds in both cases, in the second one by the condition $0<b(0)\leq b(1)\leq\dots$ from Definition~\ref{def:Goodstein-system}. The latter does now provide~$i,j\in Y$ with $i<j$ and $c_{ij}(d(i))\leq d(j)$. Writing $i=y(k)$ and $j=y(l)$ we get $k<l$ and
\begin{equation*}
f(k)=e_i(d(i))=e_j\circ c_{ij}(d(i))\leq_X e_j(d(j))=f(l),
\end{equation*}
as required to show that~$X$ is well founded. Aiming at~(i), we now consider an element~$m\in D(b(0))$ and the resulting Goodstein sequence~$G_{b,c,m}^D(0),G_{b,c,m}^D(1),\dots$ determined by Definition~\ref{def:gen-Goodstein-sequence}. We write $\mu:D\Rightarrow\overline D\!\restriction\!\nat$ for the inverse of the natural transformation $\eta^D$ from Proposition~\ref{prop:dil-extend}. Let us define $g:\mathbb N\to\overline D(X)$ by
\begin{equation*}
g(i):=\overline D(e_i)\circ\mu_{b(i)}(G^D_{b,c,m}(i)).
\end{equation*}
If we have $G^D_{b,c,m}(i)\neq 0$, then Definition~\ref{def:gen-Goodstein-sequence} yields
\begin{equation*}
G^D_{b,c,m}(i+1)=D(c_i)(G^D_{b,c,m}(i))-1<_{D(b(i+1))}D(c_i)(G^D_{b,c,m}(i)).
\end{equation*}
Using the naturality of $\mu$ and the equation $e_{i+1}\circ c_i=e_i$ from the commutative diagram above, we can deduce
\begin{multline*}
g(i+1)<_X\overline D(e_{i+1})\circ\mu_{b(i+1)}\circ D(c_i)(G^D_{b,c,m}(i))=\\
=\overline D(e_{i+1})\circ\overline D(c_i)\circ\mu_{b(i)}(G^D_{b,c,m}(i))=\overline D(e_i)\circ\mu_{b(i)}(G^D_{b,c,m}(i))=g(i),
\end{multline*}
still under the assumption $G^D_{b,c,m}(i)\neq 0$. Statement~(ii) of the theorem ensures that $\overline D(X)$ is well founded. Hence we cannot have $g(i+1)<_X g(i)$ for all~$i\in\mathbb N$. It follows that $G^D_{b,c,m}(i)=0$ must hold eventually, as required for statement~(i).
\end{proof}

Putting things together, we can deduce the following result, which was also stated in the introduction:

\begin{theorem}\label{thm:Goodstein-ACA}
The following are equivalent over~$\rca_0$:
\begin{enumerate}[label=(\roman*)]
\item arithmetical comprehension,
\item the extended Goodstein theorem for the binary representation: for any Goodstein system $(b,c)$ and any $m<2^{b(0)}$ there is an $i\in\mathbb N$ with $G_{b,c,m}^{2}(i)=0$.
\end{enumerate}
\end{theorem}

Concerning the Goodstein sequences $G_{b,c,m}^{2}(0),G_{b,c,m}^{2}(1),\dots$ in statement~(ii), we point out that we have given two different but equivalent definitions: The ad hoc construction from the introduction and the construction that results from the general Definition~\ref{def:gen-Goodstein-sequence} for the Goodstein dilator~$2:\nat\to\lo$ defined in Example~\ref{ex:2-Goodstein}.

\begin{proof}
By Theorem~\ref{thm:Goodstein-dilator}, statement~(ii) of the present theorem is equivalent to the assertion that $\overline 2:\lo\to\lo$ preserves well foundedness, where $2:\nat\to\lo$ is the Goodstein dilator from Example~\ref{ex:2-Goodstein}. As we have seen in Example~\ref{ex:2-Goodstein-extend}, this assertion is itself equivalent to the following statement: If $X=(X,<_X)$ is a well order, then so is the lexicographic order on
\begin{equation*}
2^X=\widehat 2(X)=\{\langle x_0,\dots,x_{k-1}\rangle\,|\, x_0,\dots,x_{k-1}\in X\text{ with }x_{k-1}<_X\dots<_X x_0\}.
\end{equation*}
This last statement is known to be equivalent to arithmetical comprehension over the base theory~$\rca_0$ (see~\cite{girard87} and~\cite[Theorem~2.6]{hirst94}).
\end{proof}

To conclude this section, we present a version of our equivalence for Goodstein sequences without coefficient changes. In the following we have $G_{b,m}^D(i)=G_{b,c,m}^D(i)$ for $c_j:b(j)\to b(j+1)$ with $c_j(n)=n$ (cf.~Definition~\ref{def:gen-Goodstein-sequence}).

\begin{theorem}\label{thm:Goodstein-Pi11}
The theory $\rca_0$ proves that the following are equivalent for any Goodstein dilator~$D$:
\begin{enumerate}[label=(\roman*)]
\item the extended Goodstein theorem without coefficient changes: for any non-decreasing function $b:\mathbb N\to\mathbb N$ with $b(0)>0$ and any number $m\in D(b(0))$, there is an $i\in\mathbb N$ with $G_{b,m}^D(i)=0$,
\item the linear order~$\overline D(\mathbb N)$ is well founded.
\end{enumerate}
\end{theorem}
\begin{proof}
It suffices to modify the proof of Theorem~\ref{thm:Goodstein-dilator} as follows: In the proof that~(i) implies~(ii), we first put $\star=0\in\mathbb N=X$. We then replace $\{\star\}\cup a_0\cup\dots\cup a_i$ by the set $\{0,\dots,b(i)-1\}$, where $b(i)\in\mathbb N$ is minimal with
\begin{equation*}
\{\star\}\cup a_0\cup\dots\cup a_i\subseteq \{0,\dots,b(i)-1\}=b(i).
\end{equation*}
As a consequence, the isomorphism $b(i)\cong\{\star\}\cup a_0\cup\dots\cup a_i$ from the proof of Theorem~\ref{thm:Goodstein-dilator} becomes the identity, so that $c_{ij}:b(i)\to b(j)$ becomes the inclusion with $c_{ij}(n)=n$ for $n<b(i)$. Let us observe that this makes $(b,c)$ a Goodstein system: For any $d:\mathbb N\to\mathbb N$ with $d(i)<b(i)$ and any infinite $Y\subseteq\mathbb N$, we can clearly find $i<j$ in $Y$ with $d(j)\geq d(i)=c_{ij}(d(i))$. The rest of the argument remains unchanged, except that $\iota_i:a_i\hookrightarrow b(i)$ has modified co-domain. In the proof that~(ii) implies~(i), we replace~$X$ by~$\mathbb N$ and consider the inclusions $e_i:b(i)\hookrightarrow\mathbb N$ as well as~$c_{ij}:b(i)\hookrightarrow b(j)$. The order $\mathbb N$ may be bigger than the direct limit over the morphisms $c_{ij}$, namely when the range of~$b:\mathbb N\to\mathbb N$ is finite. Nevertheless, the rest of the argument goes through unchanged.
\end{proof}

\section{From Ackermann function to Veblen hierarchy}\label{sect:Ackermann-to-Veblen}

In this section we define a Goodstein dilator~$A:\nat\to\lo$ that is based on the Ackermann function. We then prove Theorem~\ref{thm:Ackermann-Veblen} from the introduction, which asserts that the extension $\overline A:\lo\to\lo$ of this Goodstein dilator is closely related to the Veblen hierarchy. Finally, we deduce Theorem~\ref{thm:Goodstein-ATR}, which shows that the extended Goodstein theorem for~$A$ is equivalent to arithmetical transfinite recursion.

Let us recall the fast-growing hierarchy of functions $F_b:\mathbb N\to\mathbb N$ for $b\in\mathbb N$, as considered in the introduction: It is given by the recursive clauses
\begin{equation*}
F_0(n)=n+1\qquad\text{and}\qquad F_{b+1}(n)=F_b^{1+n}(n),
\end{equation*}
where the superscript refers to iteration (which is recursively defined by $F^0(n)=n$ and $F^{m+1}(n)=F(F^m(n))$). We are particularly interested in the map~$b\mapsto F_b(1)$, a variant of the well-known Ackermann function.

As mentioned in the introduction, the Ackermann function grows so fast that $\rca_0$ cannot prove that it is total (i.\,e.~that all evaluations according to the recursive clauses will terminate). Somewhat informally, the issue is that $F_b(1)$ depends on the value of $F_{b-1}$ on the large argument $F_{b-1}(1)$, or more generally: that there is no a priori bound on the arguments that appear in a recursive evaluation. On the other hand, it is easy to prove the totality of $F_b$ by induction on~$b\in\mathbb N$. Here the induction hypothesis secures the infinitely many evaluations of $F_{b-1}(0),F_{b-1}(1),\dots$, which avoids the need for a bound. In order to accommodate this induction, we extend our base theory by the principle~$\isigma_2$ of $\Sigma^0_2$-induction. The latter is equivalent to $\Pi^0_2$-induction and strictly weaker than the principle of arithmetical comprehension, which we have encountered above. Let us stress that the additional induction principle is only needed once (to prove that the functions $F_b$ are total) and without set parameters. Also, many results could be reproduced in~$\rca_0$ itself if values of the Ackermann function were represented by terms, rather than computed as numbers. However, this would make it necessary to reformulate clause~(i) of Definition~\ref{def:Goodstein-dilator} (in terms of a function $D(b)\backslash\{0\}\ni\sigma\mapsto\sigma-1\in D(b)$ on term representations). We also think that Goodstein's theorem is more natural when Goodstein sequences consist of actual numbers. Having discussed these foundational issues, we record some standard facts about the fast growing hierarchy:

\begin{lemma}\label{lem:F-basic}
The following is provable in $\rca_0+\isigma_2$:
\begin{enumerate}[label=(\alph*)]
\item We have $n<F_b(n)$ for all $b,n\in\mathbb N$.
\item The function $F_b$ is strictly increasing for any~$b\in\mathbb N$.
\item The function $b\mapsto F_b(n)$ is strictly increasing for any~$n>0$.
\end{enumerate}
\end{lemma}
\begin{proof}
Parts~(a) and~(b) are readily verified by induction on~$b$. Part~(c) follows as
\begin{equation*}
F_{b+1}(n)=F_b^{1+n}(n)\geq F^2_b(n)>F_b(n),
\end{equation*}
where the inequalities rely on the assumption $n>0$ and on part~(a).
\end{proof}

As in the case of the binary representation (cf.~the introduction and Example~\ref{ex:2-Goodstein}), it will be crucial to have a suitable notion of normal form. The following Ackermann normal forms were previously considered by T.~Arai~\cite[Section~4.3.2.3]{arai-book-2020}.

\begin{definition}[$\rca_0+\isigma_2$]\label{def:Ack-normal-form}
By an Ackermann normal form of a number~$m\in\mathbb N$ we mean a representation
\begin{equation*}
m\nf F_{b_{k-1}}^{1+n_{k-1}}\circ\dots\circ F_{b_0}^{1+n_0}(1)
\end{equation*}
with $b_{k-1}<\dots<b_0$ and $n_i<F_{b_{i-1}}^{1+n_{i-1}}\circ\dots\circ F_{b_0}^{1+n_0}(1)$ for all~$i<k$ (in particular $n_0=0$ in case $k\neq 0$, since the empty composition is the identity).
\end{definition}

We recycle the notation from Example~\ref{ex:2-Goodstein-extend} and write $<_{\widehat 2(X)}$ for the lexicographic order between descending sequences in~$X$. In the following,  $X=\mathbb N\times\mathbb N$ carries the usual product order (cf.~the paragraph before Theorem~\ref{thm:Ackermann-Veblen} in the introduction).

\begin{proposition}[$\rca_0+\isigma_2$]\label{prop:Ack-NF}
Every number $m>0$ has a unique Ackermann normal form. Furthermore, we have
\begin{multline*}
F_{b_{k-1}}^{1+n_{k-1}}\circ\dots\circ F_{b_0}^{1+n_0}(1)<F_{b_{l-1}'}^{1+n_{l-1}'}\circ\dots\circ F_{b_0'}^{1+n_0'}(1)\quad\Leftrightarrow\\
\langle (b_0,n_0),\dots,(b_{k-1},n_{k-1})\rangle<_{\widehat 2(\mathbb N\times\mathbb N)}\langle (b_0',n_0'),\dots,(b_{l-1}',n_{l-1}')\rangle
\end{multline*}
if the numbers in the left (i.\,e.~upper) inequality are in Ackermann normal form.
\end{proposition}
\begin{proof}
In order to find an Ackermann normal form of $m>0$, we construct a sequence $1=:m_0<\dots<m_k=m$ by the following recursion: Assume that we have $m_i<m$ and hence $F_0(m_i)=m_i+1\leq m$. We then set $m_{i+1}:=F_{b_i}^{1+n_i}(m_i)\leq m$ with $(b_i,n_i)\in\mathbb N\times\mathbb N$ as large as possible, which is justified by Lemma~\ref{lem:F-basic}. More explicitly, we first maximize $b_i$ by stipulating $F_{b_i}(m_i)\leq m<F_{b_i+1}(m_i)$. Having fixed~$b_i$, we maximize $n_i$ analogously. By induction we now show that all numbers in our sequence are in Ackermann normal form, as expressed by
\begin{equation*}
m_i\nf F_{b_{i-1}}^{1+n_{i-1}}\circ\dots\circ F_{b_0}^{1+n_0}(1).
\end{equation*}
For $i=k$ this will yield the desired normal form of $m=m_k$. In view of the condition $b_{k-1}<\dots<b_0$, we will also be able to infer that the recursion terminates after~$k\leq b_0+1$ steps (which is not strictly required but still illuminating). It remains to carry out the induction. We first observe that $m_0\nf 1$ is in normal form, since the empty composition is the identity. For $i=1$ it suffices to observe that we have $n_0=0<1$, in view of $F^2_{b_0}(1)=F_{b_0+1}(1)$ and the maximality of~$b_0$. In the induction step from $m_i$ to $m_{i+1}$ with $i>0$, we need to establish $b_i<b_{i-1}$ as well as $n_i<m_i$. If the latter was false, then we would get
\begin{equation*}
F_{b_i+1}(m_i)=F_{b_i}^{1+m_i}(m_i)\leq F_{b_i}^{1+n_i}(m_i)=m_{i+1}\leq m,
\end{equation*}
contradicting the maximality of~$b_i$. Similarly, $b_i\geq b_{i-1}$ would yield
\begin{equation*}
F_{b_{i-1}}^{1+n_{i-1}+1}(m_{i-1})=F_{b_{i-1}}(m_i)\leq F_{b_i}^{1+n_i}(m_i)\leq m,
\end{equation*}
which contradicts the maximality of $n_{i-1}$ in the previous step of the construction. This completes the induction and hence the proof that normal forms exist. Since the lexicographic order is trichotomous, the uniqueness of normal forms and the direction~$\Rightarrow$ of the equivalence in the proposition reduce to the direction~$\Leftarrow$ of the same equivalence. In order to prove the latter, we consider an inequality
\begin{equation*}
\langle (b_0,n_0),\dots,(b_{k-1},n_{k-1})\rangle<_{\widehat 2(\mathbb N\times\mathbb N)}\langle (b_0',n_0'),\dots,(b_{l-1}',n_{l-1}')\rangle
\end{equation*}
between sequences that correspond to normal forms. The case where the right sequence extends the left is straightforward by Lemma~\ref{lem:F-basic}. Let us now assume that the inequality holds because we have $(b_j,n_j)<_{\mathbb N\times\mathbb N}(b_j',n_j')$ and all pairs with smaller index coincide. Writing $m_0:=F_{b_{j-1}}^{1+n_{j-1}}\circ\dots\circ F_{b_0}^{1+n_0}(1)$, it suffices to show
\begin{equation*}
F_{b_{k-1}}^{1+n_{k-1}}\circ\dots\circ F_{b_j}^{1+n_j}(m_0)<F_{b_j'}^{1+n_j'}(m_0).
\end{equation*}
Note that we have $n_j<m_0$, since we are concerned with normal forms. If the inequality $(b_j,n_j)<_{\mathbb N\times\mathbb N}(b_j',n_j')$ holds because we have $b_j<b_j'$, then we get
\begin{equation*}
F_{b_j}^{1+n_j+1}(m_0)\leq F_{b_j}^{1+m_0}(m_0)=F_{b_j+1}(m_0)\leq F_{b_j'}^{1+n_j'}(m_0).
\end{equation*}
The inequality between the outer terms (i.\,e.~without the intermediate steps) is also true when $(b_j,n_j)<_{\mathbb N\times\mathbb N}(b_j',n_j')$ holds because of~$b_j=b_j'$ and $n_j<n_j'$. By induction from~$i=j+1$ up to $i=k$, we shall now show
\begin{equation*}
F_{b_{i-1}}\circ F_{b_{i-1}}^{1+n_{i-1}}\circ\dots\circ F_{b_j}^{1+n_j}(m_0)\leq F_{b_j}^{1+n_j+1}(m_0).
\end{equation*}
For $i=k$, the left side is strictly bigger than in the desired inequality above, so that we get a strict inequality in the latter. In the induction step, the definition of normal forms yields $n_i<F_{b_{i-1}}^{1+n_{i-1}}\circ\dots\circ F_{b_j}^{1+n_j}(m_0)$ and hence
\begin{equation*}
F_{b_i}\circ F_{b_i}^{1+n_i}\circ F_{b_{i-1}}^{1+n_{i-1}}\circ\dots\circ F_{b_j}^{1+n_j}(m_0)\leq F_{b_i+1}\circ F_{b_{i-1}}^{1+n_{i-1}}\circ\dots\circ F_{b_j}^{1+n_j}(m_0).
\end{equation*}
In view of $b_i<b_{i-1}$ we can conclude by the induction hypothesis.
\end{proof}

Using our notion of normal form, we will now define a Goodstein dilator based on the Ackermann function. As preparation, we observe that Proposition~\ref{prop:Ack-NF} entails
\begin{equation*}
m\nf F_{b_{k-1}}^{1+n_{k-1}}\circ\dots\circ F_{b_0}^{1+n_0}(1)<F_b(1)\quad\Leftrightarrow\quad b_0<b\text{ or }k=0,
\end{equation*}
due to $F_b(1)\nf F_b^{1+0}(1)$. For $m$ as in the equivalence, we also note that $n_i<m$ holds for all $i<k$, which justifies induction and recursion over normal forms.

\begin{definition}[$\rca_0+\isigma_2$]\label{def:Ack-dil}
For each number~$b\in\mathbb N$ we put
\begin{equation*}
A(b):=\{0,\dots,F_b(1)-1\}.
\end{equation*}
Given a strictly increasing function $f:b=\{0,\dots,b-1\}\to\{0,\dots,b'-1\}=b'$, we define a function $A(f)$ with domain $A(b)$ by recursion, setting $A(f)(0)=0$ and
\begin{equation*}
A(f)\left(F_{b_{k-1}}^{1+n_{k-1}}\circ\dots\circ F_{b_0}^{1+n_0}(1)\right)=F_{f(b_{k-1})}^{1+A(f)(n_{k-1})}\circ\dots\circ F_{f(b_0)}^{1+A(f)(n_0)}(1),
\end{equation*}
where the argument is assumed to be in Ackermann normal form.
\end{definition}

Let us establish the promised result:

\begin{proposition}[$\rca_0+\isigma_2$]\label{prop:Ack-Goodstein-dil}
The constructions from Definition~\ref{def:Ack-dil} yield a Goodstein dilator~$A:\nat\to\lo$.
\end{proposition}
\begin{proof}
As a first major step, we show that $A(f)$ is order preserving for an arbitrary morphism $f:b\to b'$ in $\nat$. For this purpose we define a length function $L:\mathbb N\to\mathbb N$ by recursion over normal forms, setting $L(0)=0$ and
\begin{equation*}
L(m)=L(n_{k-1})+\dots+L(n_0)+k\quad\text{for}\quad m\nf F_{b_{k-1}}^{1+n_{k-1}}\circ\dots\circ F_{b_0}^{1+n_0}(1).
\end{equation*}
In order to show that $m<m'<F_b(1)$ entails $A(f)(m)<A(f)(m')$, we now argue by induction over $L(m)+L(m')$. For $m$ as in the definition of~$L$, the value
\begin{equation*}
A(f)(m)\nf F_{f(b_{k-1})}^{1+A(f)(n_{k-1})}\circ\dots\circ F_{f(b_0)}^{1+A(f)(n_0)}(1)
\end{equation*}
is still in normal form. Indeed, the crucial inequalities $n_i<F_{b_{i-1}}^{1+n_{i-1}}\circ\dots\circ F_{b_0}^{1+n_0}(1)$ for $i<k$ are preserved by the induction hypothesis, which applies due to
\begin{equation*}
L(n_i)+L(F_{b_{i-1}}^{1+n_{i-1}}\circ\dots\circ F_{b_0}^{1+n_0}(1))=L(n_i)+\dots+L(n_0)+i<L(m).
\end{equation*}
A given normal form of $m'$ is preserved in the same way. Combining this fact with Proposition~\ref{prop:Ack-NF} and the induction hypothesis, it is straightforward to deduce $A(f)(m)<A(f)(m')$ from~$m<m'$. As part of our inductive argument, we have shown that $A(f)$ preserves normal forms, which is central for much of the following. In particular, we can now observe that $A(f)$ has values in $A(b')$, still for $f:b\to b'$ (cf.~the paragraph before Definition~\ref{def:Ack-dil}). We can also deduce that $A$ is functorial, by a straightforward induction. To conclude that~$A$ is a Goodstein dilator, we must show that it preserves pullbacks. With Proposition~\ref{prop:pullback-nat-transf} in mind, we define functions $\supp_b:A(b)\to[b]^{<\omega}$ by recursion over normal forms, setting $\supp_b(0)=\emptyset$ and
\begin{equation*}
\supp_b(F_{b_{k-1}}^{1+n_{k-1}}\circ\dots\circ F_{b_0}^{1+n_0}(1))=\{b_0,\dots,b_{k-1}\}\cup\supp_b(n_0)\cup\dots\cup\supp_b(n_{k-1}).
\end{equation*}
Given that the morphisms $A(f)$ preserve normal forms, a straightforward induction shows that $\supp:A\to[\cdot]^{<\omega}$ is a natural transformation. For $f:b'\to b$, the required implication
\begin{equation*}
\supp_b(m)\subseteq\rng(f)\quad\Rightarrow\quad m\in\rng(A(f))
\end{equation*}
can be shown by induction over the normal form of $m\nf F_{b_{k-1}}^{1+n_{k-1}}\circ\dots\circ F_{b_0}^{1+n_0}(1)$. For $i<k$ the induction hypothesis allows us to write $n_i=A(f)(n_i')$ with $n_i'\in A(b')$. We also get $b_i=f(b_i')$ with $b_i'<b'$. By induction on~$i$ one can show that
\begin{equation*}
m'_i:=F_{b'_{i-1}}^{1+n'_{i-1}}\circ\dots\circ F_{b'_0}^{1+n'_0}(1)\in A(b')
\end{equation*}
is in normal form. The induction hypothesis allows us to compute $A(f)(m'_i)$ as
\begin{equation*}
A(f)(m'_i)=F_{b_{i-1}}^{1+n_{i-1}}\circ\dots\circ F_{b_0}^{1+n_0}(1)>n_i=A(f)(n_i'),
\end{equation*}
where the inequality comes from the normal form condition for~$m$. We can infer $n'_i<m'_i$, as required for the induction step. For $i=k$, the given normal form of~$m_k'\in A(b')$ reveals $m=A(f)(m'_k)\in\rng(A(f))$, as desired. By Proposition~\ref{prop:pullback-nat-transf}, we can now conclude that $A$ preserves pullbacks and is thus a Goodstein dilator.
\end{proof}

Now that $A:\nat\to\lo$ is defined as a Goodstein dilator, Definition~\ref{def:gen-Goodstein-sequence} yields Goodstein sequences $G^A_{b,c,m}(0),G^A_{b,c,m}(1),\dots$ that are based on the Ackermann function. By Theorem~\ref{thm:Goodstein-dilator}, the corresponding Goodstein theorem is closely related to the extension $\overline A:\lo\to\lo$. In the following we investigate the latter. As we have seen in the previous section, the general construction of extensions (according to Definition~\ref{def:dil-extend}) can be hard to understand in concrete cases (cf.~Example~\ref{ex:2-Goodstein-extend} and the paragraph that preceeds it). Rather than pondering over the definition of~$\overline A$, we will thus give an ad hoc definition of a plausible extension~$\widehat A:\lo\to\lo$. Using Proposition~\ref{prop:determine-extension}, we will then confirm that $\widehat A$ and $\overline A$ are isomorphic. Since the construction of $\widehat A$ will be based on term representations (rather than numerical values of the Ackermann function), it can be implemented in~$\rca_0$. However, we will need to switch back to the stronger base theory $\rca_0+\isigma_2$ in order to consider the isomorphism $\widehat A\cong\overline A$.

\begin{definition}[$\rca_0$]\label{def:hat-A-objects}
For each linear order~$X$, we define a set $\widehat A(X)$ with a binary relation $<_{\widehat A(X)}$ by simultaneous recursion (cf.~the justification below):
\begin{enumerate}[label=(\roman*)]
\item The term $0$ is an element of~$\widehat A(X)$.
\item Given elements $x_{k-1}<_X\dots<_X x_0$ of $X$ and terms $s_0,\dots,s_{k-1}\in\widehat A(X)$, we add terms $\chi_{x_{j-1}}^{s_{j-1}}\circ\dots\circ\chi_{x_0}^{s_0}(1)\in\widehat A(X)$ for $j\leq k$, as long as we have $s_i<_{\widehat A(X)}\chi_{x_{i-1}}^{s_{i-1}}\circ\dots\circ\chi_{x_0}^{s_0}(1)$ for all $i<j$ (in particular~$1\in\widehat A(X)$).
\end{enumerate}
The term $0$ is the minimum element of $\widehat A(X)$, and we have
\begin{multline*}
\chi_{x_{k-1}}^{s_{k-1}}\circ\dots\circ\chi_{x_0}^{s_0}(1)<_{\widehat A(X)} \chi_{y_{l-1}}^{t_{l-1}}\circ\dots\circ\chi_{y_0}^{t_0}(1)\Leftrightarrow\\
\langle (x_0,s_0),\dots,(x_{k-1},s_{k-1})\rangle<_{\widehat 2(X\times\widehat A(X))} \langle (y_0,t_0),\dots,(y_{l-1},t_{l-1})\rangle,
\end{multline*}
where $<_{\widehat 2(X\times\widehat A(X))}$ is the lexicographic order with respect to the usual product order on the set $X\times\widehat A(X)$ (cf.~the paragraph before Theorem~\ref{thm:Ackermann-Veblen} in the introduction). Given an order embedding $f:X\to Y$, we define a function $\widehat A(f):\widehat A(X)\to\widehat A(Y)$ (cf.~the discussion below) by recursion over terms, setting $\widehat A(f)(0)=0$ and
\begin{equation*}
\widehat A(f)\left(\chi_{x_{k-1}}^{s_{k-1}}\circ\dots\circ\chi_{x_0}^{s_0}(1)\right)=\chi_{f(x_{k-1})}^{\widehat A(f)(s_{k-1})}\circ\dots\circ\chi_{f(x_0)}^{\widehat A(f)(s_0)}(1).
\end{equation*}
\end{definition}

To justify the given recursion in detail, we first construct sets $\widehat A^0(X)\supseteq\widehat A(X)$ by ignoring the condition $s_i<_{\widehat A(X)}\chi_{x_{i-1}}^{s_{i-1}}\circ\dots\circ\chi_{x_0}^{s_0}(1)$ in~(ii). We then define a length function $L:\widehat A(X)^0\to\mathbb N$ by setting $L(0)=0$ and
\begin{equation*}
L\left(\chi_{x_{k-1}}^{s_{k-1}}\circ\dots\circ\chi_{x_0}^{s_0}(1)\right)=L(s_{k-1})+\dots+L(s_0)+k.
\end{equation*}
Let us point out the analogy with the proof of Proposition~\ref{prop:Ack-Goodstein-dil} (which justifies the use of the same letter~$L$). For terms $s,t_0,t_1\in\widehat A^0 (X)$ we can now decide $s\in\widehat A(X)$ and $t_0<_{\widehat A(X)}t_1$ by simultaneous recursion on $L(s)$ and $L(t_0)+L(t_1)$, respectively. For $f:X\to Y$, the values of $\widehat A(f)$ will certainly lie in~$\widehat A^0(Y)$, but it is not immediately clear whether they lie in~$\widehat A(Y)$. As part of the following proof, we show that they do.

\begin{proposition}[$\rca_0$]
Definition~\ref{def:hat-A-objects} yields a functor~$\widehat A:\lo\to\lo$.
\end{proposition}
\begin{proof}
To show that $<_{\widehat A(X)}$ is linear for each linear order~$X$, one checks $s\not<_{\widehat A(X)}s$ by induction on $L(s)$, then $s<_{\widehat A(X)}t\,\lor\, s=t\,\lor\, t<_{\widehat A(X)}s$ by induction on $L(s)+L(t)$, and finally $r<_{\widehat A(X)}s\,\land\, s<_{\widehat A(X)} t\,\to\, r<_{\widehat A(X)}t$ by induction on $L(r)+L(s)+L(t)$. Essentially, this amounts to the usual proof that the lexicographic order with respect to a linear order~$Z$ is linear itself, except that the assumption on~$Z$ is replaced by the induction hypothesis. Now consider a morphism $f:X\to Y$ in~$\lo$, i.\,e.~an order embedding. To establish the implications
\begin{align*}
s\in\widehat A(X)\,&\to\,\widehat A(f)(s)\in\widehat A(Y),\\
t_0<_{\widehat A(X)}t_1\,&\to\,\widehat A(f)(t_0)<_{\widehat A(Y)}\widehat A(f)(t_1),
\end{align*}
one argues by simultaneous induction on $L(s)$ and~$L(t_0)+L(t_1)$, respectively. The point is that the inequalities $s_i<_{\widehat A(X)}\chi_{x_{i-1}}^{s_{i-1}}\circ\dots\circ\chi_{x_0}^{s_0}(1)$ required in clause~(ii) of Definition~\ref{def:hat-A-objects} are preserved due to the simultaneous induction hypothesis. Functoriality is readily established by induction over the build-up of terms.
\end{proof}

The following was promised (and explained) before Definition~\ref{def:hat-A-objects}.

\begin{proposition}[$\rca_0+\isigma_2$]\label{prop:iso-A-hat-bar}
There is a natural isomorphism $\widehat A\cong\overline A$.
\end{proposition}
\begin{proof}
We use the criterion from Proposition~\ref{prop:determine-extension}. Condition~(i) of the latter requires an isomorphism $\mu:\widehat A\!\restriction\!\nat\Rightarrow A$. For each object $b=\{0,\dots,b-1\}$ in $\nat$ we define the component $\mu_b:\widehat A(b)\to A(b)$ by recursion over terms: Set $\mu_b(0)=0$ and
\begin{equation*}
\mu_b\left(\chi_{b_{k-1}}^{s_{k-1}}\circ\dots\circ\chi_{b_0}^{s_0}(1)\right)=F_{b_{k-1}}^{1+\mu_b(s_{k-1})}\circ\dots\circ F_{b_0}^{1+\mu_b(s_0)}(1),
\end{equation*}
where the right side is evaluated as a natural number (which is possible in our strengthened base theory). To show that $s<_{\widehat A(b)}t$ implies $\mu_b(s)<\mu_b(t)$, one argues by induction on~$L(s)+L(t)$. Crucially, the induction hypothesis ensures that the recursive clause provides $\mu_b(s)$ and $\mu_b(t)$ in Ackermann normal form. This fact has several important consequences: First, the induction step can now be completed by a straightforward application of Proposition~\ref{prop:Ack-NF}. Secondly, the latter also ensures that $\mu_b$ takes values in \mbox{$A(b)=\{0,\dots,F_b(1)-1\}$}, as implicitly claimed above. Finally, we can deduce that $\mu$ is natural, by a straightforward induction over terms. To conclude that $\mu$ is a natural isomorphism, we show that each number
\begin{equation*}
m\nf F^{1+n_{k-1}}_{b_{k-1}}\circ\dots\circ F^{1+n_0}_{b_0}(1)\in\{0,\dots,F_b(1)-1\}=A(b)
\end{equation*}
lies in the range of the component~$\mu_b$. Arguing by induction over the numerical value of~$m$, we may write $n_i=\mu_b(s_i)$ with $s_i\in\widehat A(b)$ for each $i<k$. By side induction on~$i$ we obtain~$\chi^{s_{i-1}}_{b_{i-1}}\circ\dots\circ\chi^{s_0}_{b_0}(1)\in\widehat A(b)$. Let us point out that the inequality $s_i<_{\widehat A(b)}\chi^{s_{i-1}}_{b_{i-1}}\circ\dots\circ\chi^{s_0}_{b_0}(1)$, which is required for the side induction step, reduces to the condition $n_i<F_{b_{i-1}}^{1+n_{i-1}}\circ\dots\circ F_{b_0}^{1+n_0}(1)$ from Definition~\ref{def:Ack-normal-form}, since~$\mu_b$ is an order embedding. For $i=k$ we get
\begin{equation*}
m=\mu_b\left(\chi^{s_{k-1}}_{b_{k-1}}\circ\dots\circ\chi^{s_0}_{b_0}(1)\right)\in\rng(\mu_b),
\end{equation*}
as desired. In order to satisfy condition~(ii) of Proposition~\ref{prop:determine-extension}, we now define functions $\widehat\supp_b:\widehat A(b)\to[b]^{<\omega}$ by the recursive clauses $\widehat\supp_b(0)=\emptyset$ and
\begin{equation*}
\widehat\supp_b\left(\chi^{s_{k-1}}_{b_{k-1}}\circ\dots\circ\chi^{s_0}_{b_0}(1)\right)=\{b_0,\dots,b_{k-1}\}\cup\widehat\supp_b(s_0)\cup\dots\cup\widehat\supp_b(s_{k-1}).
\end{equation*}
Given a morphism $f:b'\to b$, the required implication
\begin{equation*}
\widehat\supp_b(s)\subseteq\rng(f)\quad\Rightarrow\quad s\in\rng(\widehat A(f))
\end{equation*}
can be established by induction on the build-up of~$s\in\widehat A(b)$, similarly to the proof of Proposition~\ref{prop:Ack-Goodstein-dil}. Now that we have verified all relevant conditions, Proposition~\ref{prop:determine-extension} provides the desired isomorphism $\widehat A\cong\overline A$.
\end{proof}

As in the previous proof, one can construct functions $\widehat\supp_X:\widehat A(X)\to[X]^{<\omega}$ for all linear orders~$X$, such that an implication as in Proposition~\ref{prop:determine-extension} is satisfied. It follows that the functor~$\widehat A:\lo\to\lo$ preserves pullbacks and direct limits, essentially by~\cite[Theorem~2.3.12]{girard-pi2} (see also~\cite[Remark~2.2.2]{freund-thesis} and Proposition~\ref{prop:pullback-nat-transf} above). Over a sufficiently strong base theory, Theorem~\ref{thm:Goodstein-ATR} (proved below) ensures that $\widehat A(X)\cong\overline A(X)$ is well founded for any well order~$X$. This means that~$\widehat A$ is a dilator in the usual sense.

Our next objective is to relate the functor~$\widehat A$ to the Veblen hierarchy of normal functions. Recall that a function~$f$ from ordinals to ordinals is normal if it is strictly increasing and continuous at limit stages, in the sense that $f(\lambda)=\sup\{f(\alpha)\,|\,\alpha<\lambda\}$ holds for any limit ordinal~$\lambda$. Equivalently, $f$ is the strictly increasing enumeration of a closed and unbounded (club) class of ordinals. If $f$ is a normal function, then the class $\{\gamma\,|\,f(\gamma)=\gamma\}$ of fixed points is itself club. The normal function that enumerates this class is called the derivative~$f'$ of~$f$. The Veblen hierarchy of functions $\varphi_\alpha$ iterates this construction along the ordinals: We have $\varphi_0(\gamma)=\omega^\gamma$, as defined in ordinal arithmetic; the function $\varphi_{\alpha+1}$ is the derivative of~$\varphi_\alpha$; and if $\lambda$ is a limit, then $\varphi_\lambda$ enumerates the class $\{\gamma\,|\,\varphi_\alpha(\gamma)=\gamma\text{ for all $\alpha<\lambda$}\}$, which is also club. Due to the construction we have $\varphi_\alpha\circ\varphi_\beta=\varphi_\beta$ for $\alpha<\beta$. Together with monotonicity, one can conclude
\begin{equation}\label{eq:ineq-Veblen}\tag{$\star$}
\varphi_\alpha(\gamma)<\varphi_\beta(\delta)\quad\Leftrightarrow\quad
\begin{cases}
\text{either} &\alpha<\beta\text{ and }\gamma<\varphi_\beta(\delta),\\
\text{or} & \alpha=\beta\text{ and }\gamma<\delta,\\
\text{or} & \alpha>\beta\text{ and }\varphi_\alpha(\gamma)<\delta.
\end{cases}
\end{equation}
This equivalence suggests a recursive definition of inequality on a suitable set of terms. In the present paper, we consider term representations for the Veblen hierarchy along a fixed linear order. Let us recall that $1+X$ represents the extension of the order~$X$ by a new minimum element, which we denote by~$0$. The idea is to define a term system $\varphi_{1+X}0$ with a binary relation $<_{\varphi_{1+X}0}$ and an auxiliary function $h:\varphi_{1+X}0\to 1+X$ by simultaneous recursion:
\begin{itemize}
\item The term~$0$ is an element of~$\varphi_{1+X}0$, and we have $h(0)=0$.
\item Given a term $s\in\varphi_{1+X}0$ and an element $x\in 1+X$ with $h(s)\leq_{1+X}x$, we add a term $\varphi_xs\in\varphi_{1+X}0$ with $h(\varphi_xs)=x$.
\item Given $n>1$ terms $\varphi_{x_{n-1}}s_{n-1}\leq_{\varphi_{1+X}0}\dots\leq_{\varphi_{1+X}0}\varphi_{x_0}s_0$ of the indicated form, we add a term $\varphi_{x_0}s_0+\dots+\varphi_{x_{n-1}}s_{n-1}=:s\in\varphi_{1+X}0$ with $h(s)=0$.
\end{itemize}
The simultaneous definition of~$<_{\varphi_{1+X}0}$ reflects inequality~(\ref{eq:ineq-Veblen}) and the idea that terms of the form $\varphi_xs$ are additively principal, in the sense that \mbox{$\varphi_{x_0}s_0<_{\varphi_{1+X}0}\varphi_xs$} entails $\varphi_{x_0}s_0+\dots+\varphi_{x_{n-1}}s_{n-1}<_{\varphi_{1+X}0}\varphi_xs$. For a complete list of recursive clauses we refer to~\cite[Section~2]{rathjen-weiermann-atr} (where $Q=1+X$ and $0_Q=0\in1+X$). In the cited reference it is shown that $\varphi_{1+X}0$ is a linear order when the same holds for~$X$, provably in~$\rca_0$. To explain the role of~$h$, we point out that $\varphi_x\varphi_y0$ and $\varphi_y0$ should have the same interpretation if $x<_{1+X}y$. In view of \mbox{$h(\varphi_y0)=y\not\leq_{1+X}x$}, the ``superfluous" term $\varphi_x\varphi_y0$ is excluded from~$\varphi_{1+X}0$. Before we relate the transformation $X\mapsto\varphi_{1+X}0$ to our functor~$\widehat A:\lo\to\lo$, we reflect on the definition of the latter:

\begin{remark}\label{rmk:chi-semantics}
In most cases, ordinal notation systems are discussed on a semantic level first: One considers set-theoretic constructions on the actual ordinals and proves characteristic properties, such as our equivalence~(\ref{eq:ineq-Veblen}). Only in a second step, these properties are reproduced on a syntactic level. Our presentation of the notation systems~$\varphi_{1+X}0$ provides an example of this approach. As another example, the Cantor normal form theorem motivates the most common notation system for the ordinal~$\varepsilon_0=\min\{\alpha\,|\,\omega^\alpha=\alpha\}$. On the other hand, the notation systems $\widehat A(X)$ from Definition~\ref{def:hat-A-objects} were introduced in a different way: There is still a semantic aspect, since we have started by investigating the fast growing hierarchy of functions~$F_b:\mathbb N\to\mathbb N$ with~$b\in\mathbb N$. However, the extension from natural numbers to arbitrary ordinals (in fact to linear orders) was based on formal syntactic rather than semantic considerations. To complete the picture, we now sketch a corresponding semantic construction, even though the latter will play no official role in the following. Given a function $\chi$ from ordinals to ordinals, one can define ordinal iterates by the recursive clauses
\begin{equation*}
\chi^0(\gamma)=\gamma,\quad\chi^{\alpha+1}(\gamma)=\chi(\chi^\alpha(\gamma)),\quad\chi^\lambda(\gamma)=\sup\{\chi^\alpha(\gamma)\,|\,\alpha<\lambda\}\text{ for $\lambda$ limit}.
\end{equation*}
We point out that the limit case is most natural when $\alpha\mapsto\chi^\alpha(\gamma)$ is weakly (or strictly) increasing, which is the case when we have $\gamma\leq\chi(\gamma)$ (or $\gamma<\chi(\gamma)$) for all ordinals~$\gamma$. We now define a hierarchy of functions $\chi_\alpha$ by stipulating
\begin{equation*}
\chi_0(\gamma)=\gamma+1,\quad\chi_{\alpha+1}(\gamma)=\chi_\alpha^{1+\gamma}(\gamma),\quad\chi_\lambda(\gamma)=\sup\{\chi_\alpha(\gamma)\,|\,\alpha<\lambda\}\text{ for $\lambda$ limit}.
\end{equation*}
Note that this extends the finite (but not the infinite) stages of the fast-growing hierarchy, in the sense that we have $\chi_b(n)=F_b(n)$ for $b,n\in\mathbb N$. The properties from Lemma~\ref{lem:F-basic} change slightly (note~$\chi_\omega(n)=\omega$ for every~$n\in\mathbb N\backslash\{0\}$): We have $\gamma<\chi_\alpha(\gamma)$, each function $\chi_\alpha$ is weakly increasing, and $\alpha\mapsto\chi_\alpha(\gamma)$ is strictly increasing (hence normal) in case~$\gamma>0$. By generalizing Definition~\ref{def:Ack-normal-form} in the obvious way, we obtain a notion of normal form
\begin{equation*}
\delta\nf\chi^{1+\gamma_{k-1}}_{\alpha_{k-1}}\circ\dots\circ\chi^{1+\gamma_0}_{\alpha_0}(1).
\end{equation*}
The analogue of Proposition~\ref{prop:Ack-NF} can be established by essentially the same proof: For $\delta'<\delta$ we can still pick $(\alpha,\gamma)$ maximal with $\chi_\alpha^{1+\gamma}(\delta')\leq\delta$, due to continuity in~$\alpha$ and~$\gamma$. Arguing as in the proof of Proposition~\ref{prop:iso-A-hat-bar}, one can deduce $\widehat A(\alpha)\cong\chi_\alpha(1)$, where the left side is explained by Definition~\ref{def:hat-A-objects} and ordinals are identified with the ordered sets of their predecessors. The reader may have noticed that we write elements of $\widehat A(\alpha)$ as $\chi^{s_{k-1}}_{\alpha_{k-1}}\circ\dots\circ\chi^{s_0}_{\alpha_0}(1)$ rather than $\chi^{1+s_{k-1}}_{\alpha_{k-1}}\circ\dots\circ\chi^{1+s_0}_{\alpha_0}(1)$, with a summand~$1$ missing in the exponents. This is pure notational convenience (and harmless, since the concrete way in which terms are displayed is irrelevant).
\end{remark}

Let us now compare the functor~$\widehat A:\lo\to\lo$ with the notation systems that represent the Veblen hierarchy. Together with Proposition~\ref{prop:iso-A-hat-bar}, the following result yields the first inequality of Theorem~\ref{thm:Ackermann-Veblen} from the introduction.

\begin{proposition}[$\rca_0$]\label{prop:A-into-phi}
For any linear order~$X$, there is an order embedding of $\widehat A(X)$ into $\varphi_{1+X}0$.
\end{proposition}
\begin{proof}
The desired embedding $o:\widehat A(X)\to\varphi_{1+X}0$ can be defined by the recursive clauses~$o(0)=\varphi_00=:1$, $o(1)=\varphi_01$ and
\begin{equation*}
o\left(\chi^{s_k}_{x_k}\circ\dots\circ\chi^{s_0}_{x_0}(1)\right)=\varphi_{x_k}\left(o\left(\chi^{s_{k-1}}_{x_{k-1}}\circ\dots\circ\chi^{s_0}_{x_0}(1)\right)+o(s_k)\right).
\end{equation*}
To show that $s<_{\widehat A(X)}t$ implies $o(s)<_{\varphi_{1+X}0} o(t)$, we use induction on $L(s)+L(t)$, where~$L$ is the length function specified after Definition~\ref{def:hat-A-objects}. Simultaneously, we verify that $o(s)$ and $o(t)$ are indeed terms in $\varphi_{1+X}0$. For a term as in the third recursive clause, this claim reduces to
\begin{equation*}
o(s_k)\leq_{\varphi_{1+X}0}o\left(\chi^{s_{k-1}}_{x_{k-1}}\circ\dots\circ\chi^{s_0}_{x_0}(1)\right),
\end{equation*}
which follows from the condition in Definition~\ref{def:hat-A-objects}, using the induction hypothesis. Concerning the inductive proof of monotonicity, we consider an inequality
\begin{equation*}
s=\chi^{s_k}_{x_k}\circ\dots\circ\chi^{s_0}_{x_0}(1)<_{\widehat A(X)}\chi^{t_l}_{y_l}\circ\dots\circ\chi^{t_0}_{y_0}(1)=t.
\end{equation*}
If $(x_i,s_i)=(y_i,t_i)$ holds for all $i\leq k<l$, one can deduce $o(s)<_{\varphi_{1+X}0}o(t)$ from the fact that $\varphi_xr\in\varphi_{1+X}0$ entails $r<_{\varphi_{1+X}0}\varphi_xr$. The latter is, for example, proved in~\cite[Lemma~3.1]{freund_predicative-collapsing}. Let us point out that a fixed point such as $\varphi_10=\varphi_0\varphi_10$ yields no counterexample, since $h(\varphi_10)=1\not\leq 0$ means $\varphi_0\varphi_10\notin\varphi_{1+X}0$. Now assume that $s<_{\widehat A(X)}t$ holds because there is a $j\leq\min\{k,l\}$ with $(x_j,s_j)<_{X\times\widehat A(X)}(y_j,t_j)$ and $\chi^{s_{j-1}}_{x_{j-1}}\circ\dots\circ\chi^{s_0}_{x_0}(1)=\chi^{t_{j-1}}_{y_{j-1}}\circ\dots\circ\chi^{t_0}_{y_0}(1)=:r$. By induction on $i=j,\dots,k$ we show
\begin{equation*}
o\left(\chi^{s_i}_{x_i}\circ\dots\circ\chi^{s_0}_{x_0}(1)\right)<_{\varphi_{1+X}}\varphi_{y_j}(o(r)+o(t_j))=o\left(\chi^{t_j}_{y_j}\circ\dots\circ\chi^{t_0}_{y_0}(1)\right)\leq_{\varphi_{1+X}}o(t).
\end{equation*}
Note that the case $i=k$ amounts to the desired inequality~$o(s)<_{\varphi_{1+X}0}o(t)$. Concerning the base~$i=j$, we first assume that $(x_j,s_j)<_{X\times\widehat A(X)}(y_j,t_j)$ holds because of~$x_j<_Xy_j$. Due to the induction hypothesis, the condition $s_j<_{\widehat A(X)}r$ from Definition~\ref{def:hat-A-objects} entails
\begin{equation*}
o(s_j)<_{\varphi_{1+X}0}o(r)\leq_{\varphi_{1+X}0}o(r)+o(t_j)<_{\varphi_{1+X}0}\varphi_{y_j}(o(r)+o(t_j)).
\end{equation*}
Since the right side is additively principal (cf.~\cite[Section~2]{rathjen-weiermann-atr} and the discussion above), it also bounds $o(r)+o(s_j)$. By equivalence~(\ref{eq:ineq-Veblen}) we get
\begin{equation*}
o\left(\chi^{s_j}_{x_j}\circ\dots\circ\chi^{s_0}_{x_0}(1)\right)=\varphi_{x_j}(o(r)+o(s_j))<_{\varphi_{1+X}0}\varphi_{y_j}(o(r)+o(t_j)),
\end{equation*}
as required. If we have $x_j=y_j$ and $s_j<_{\widehat A(X)}t_j$, then the induction hypothesis yields $o(r)+o(s_j)<_{\varphi_{1+X}0}o(r)+o(t_j)$, and it is straightforward to conclude by~(\ref{eq:ineq-Veblen}). The step from~$i$ to~$i+1$ is similar to the first part of the case~$i=j$, since clause~(ii) of Definition~\ref{def:hat-A-objects} entails that we have $x_{i+1}<_X x_j\leq y_j$. 
\end{proof}

In the rest of this paper we prove the other inequality from Theorem~\ref{thm:Ackermann-Veblen}, which demands an embedding of~$\varphi_{1+X}(0)$ into $\widehat A((2+X)\times\mathbb N)$. Let us recall that $2+X$ is the extension of $X$ by two bottom elements, which we denote by~$-1$ and~$0$ (with $-1<_{2+X}0<_{2+X}x$ for any $x\in X$). Whenever $2+X$ and $1+X$ appear in the same context, we identify the latter with $(2+X)\backslash\{-1\}$ (so that the bottom element of $1+X$ is denoted by~$0$, as before). The following definition yields a natural refinement of the function $h:\varphi_{1+X}0\to 1+X$ that was defined simultaneously with the set~$\varphi_{1+X}0$. For $s\in\varphi_{1+X}0$, we also define the ``most significant" subterm~$T(s)$ and the remainder~$R(s)$. Note that the tuple $(H(s),T(s),R(s))$ determines~$s$ uniquely.

\begin{definition}[$\rca_0$]\label{def:coeff-subterms}
Let $H:\varphi_{1+X}0\to 2+X$ be given by $H(0)=-1$ and
\begin{align*}
H(\varphi_xs)&=x&&\text{(which lies in $1+X=(2+X)\backslash\{-1\}$)},\\
H(\varphi_{x_0}s_0+\dots+\varphi_{x_{n-1}}s_{n-1})&=-1&&(\text{for $n>1$}).
\end{align*}
We also define functions $T,R:\varphi_{1+X}0\to\varphi_{1+X}0$, by setting $T(0)=0=R(0)$, $T(\varphi_xs)=s$, $R(\varphi_xs)=0$ and
\begin{align*}
T(\varphi_{x_0}s_0+\dots+\varphi_{x_{n-1}}s_{n-1})&=\max(\{\varphi_{x_0}s_0\}\cup\{\varphi_{x_1}s_1+\dots+\varphi_{x_{n-1}}s_{n-1}\}),\\
R(\varphi_{x_0}s_0+\dots+\varphi_{x_{n-1}}s_{n-1})&=\min(\{\varphi_{x_0}s_0\}\cup\{\varphi_{x_1}s_1+\dots+\varphi_{x_{n-1}}s_{n-1}\}),
\end{align*}
where maximum and minimum are taken in the order~$\varphi_{1+X}0$.
\end{definition}

Let us recall that the lower indices of a term $\chi^{s_k}_{x_k}\circ\dots\circ\chi^{s_0}_{x_0}(1)\in\widehat A(X)$ must be strictly increasing. Given~$s\in\varphi_{1+X}0$, this suggests to look for the first subterm~$t$ with $H(s)<_{2+X} H(t)$. To make this precise, we consider the iterates $T^n(s)$ that are recursively explained by $T^0(s)=s$ and $T^{n+1}(s)=T(T^n(s))$. Note that we must reach $T^n(s)=0$ for sufficiently large~$n$, since $T(t)$ is a proper subterm of~$t$ when the latter is different from~$0$. If the reader thinks that our notion of subterm is too imprecise, they may alternative use the length function from the proof of Proposition~\ref{prop:o-into-A} below. In any case, the fact that our iterates reach~$T^n(s)=0$ justifies the following construction:

\begin{definition}[$\rca_0$]
The function $T_*:\varphi_{1+X}0\to\varphi_{1+X}0$ is given by
\begin{equation*}
T_*(s)=T^{n(s)}(s)\quad\text{with}\quad n(s)=\min\{n\in\mathbb N\,|\,H(s)<_{2+X}H(T^n(s))\text{ or }T^n(s)=0\}.
\end{equation*}
To define $H_*:\varphi_{1+X}0\to\mathbb N$, we declare that
\begin{equation*}
H_*(s)=|\{n<n(s)\,|\,H(T^n(s))=H(s)\}|
\end{equation*}
is the number of iterates $T^n(s)$ with $n<n(s)$ for which $H(T^n(s))$ is as big as possible. Let us also set
\begin{equation*}
N(s)=\min\{N\in\mathbb N\,|\,T_*^N(s)=0\},
\end{equation*}
where $T_*^N=(T_*)^N$ denotes the $N$-th iterate of~$T_*$. Finally, we abbreviate
\begin{equation*}
s[i]=T_*^{N(s)-i}(s)
\end{equation*}
for any $i\leq N(s)$.
\end{definition}

To justify the definition of~$N(s)$, we point out that $t\neq 0$ entails $n(t)>0$, so that~$T_*(t)$ is a proper subterm of~$t$. The following result will play a fundamental role. It provides first evidence that our construction relates to the lexicographic order from the definition of~$\widehat A$.

\begin{lemma}[$\rca_0$]\label{lem:star-extend}
For each~$s\in\varphi_{1+X}0$ we have $N(s)\leq N(T(s))+1$ and
\begin{equation*}
s[i]=T(s)[i]\quad\text{for all $i<N(s)$}.
\end{equation*}
\end{lemma}
\begin{proof}
In case $T_*(s)=0$ we have $N(s)\leq 1$ and the claim holds since $t[0]=T_*^{N(t)}(t)$ is always equal to~$0$. For the rest of the proof we assume $T_*(s)\neq 0$, which entails $s\neq 0$ and hence $n(s)>0$. By induction on $k<n(s)$ we show that there is an $n\leq k$ with $T_*^n(T^{n(s)-k}(s))=T_*(s)$. For $k=0$ we recall $T_*^0(T^{n(s)}(s))=T^{n(s)}(s)=T_*(s)$. For $k>0$, the minimality of~$n(s)$ yields $T^{n(s)-k}(s)\neq 0$ and
\begin{equation*}
H(T^{n(s)-k}(s))\leq_{2+X}H(s)<_{2+X}H(T^{n(s)}(s))=H(T^k(T^{n(s)-k}(s))),
\end{equation*}
where the second inequality relies on~$T_*(s)\neq 0$. For $l:=n(T^{n(s)-k}(s))$, it follows that we have $0<l\leq k$. By induction hypothesis we get a number $m\leq k-l$ with $T_*^m(T^{n(s)-k+l}(s))=T_*(s)$. For $n:=m+1\leq k$ we obtain
\begin{multline*}
T_*^n(T^{n(s)-k}(s))=T_*^m(T_*(T^{n(s)-k}(s)))=T_*^m(T^l(T^{n(s)-k}(s)))=\\
=T_*^m(T^{n(s)-k+l}(s))=T_*(s),
\end{multline*}
as needed to complete the induction step. For $k=n(s)-1$, the result of the induction yields $T_*^n(T(s))=T_*(s)$ for some~$n$. In view of the assumption $T_*(s)\neq 0$ we can conclude $N(T(s))=n-1+N(s)$. For $i<N(s)$ we also get
\begin{multline*}
T(s)[i]=T_*^{N(T(s))-i}(T(s))=T_*^{N(s)-i-1+n}(T(s))=\\
=T_*^{N(s)-i-1}(T_*^n(T(s)))=T_*^{N(s)-i-1}(T_*(s))=T_*^{N(s)-i}(s)=s[i],
\end{multline*}
as desired.
\end{proof}

Let us also observe that we have found a monotone sequence of indices:

\begin{lemma}[$\rca_0$]\label{lem:indices-monotone}
For any $s\in\varphi_{1+X}0$ we have
\begin{equation*}
H(s[1])>_{2+X}\dots>_{2+X}H(s[N(s)])=H(s).
\end{equation*}
\end{lemma}
\begin{proof}
The equality in the lemma holds since we have $s[N(s)]=T_*^0(s)=s$. For $0<i<N(s)$ we have $s[i]\neq 0$, so that the definition of $T_*(s[i+1])=s[i]$ yields the inequality $H(s[i+1])<_{2+X}H(T_*(s[i+1]))=H(s[i])$.
\end{proof}

We now split each index $H(s)$ into three:

\begin{definition}[$\rca_0$]\label{def:split-indices}
For $i\leq 2$ we define $H_i:\varphi_{1+X}0\to(2+X)\times\mathbb N$ by
\begin{equation*}
H_0(s)=(H(s),0),\quad H_1(s)=(H(s),1),\quad H_2(s)=(H(s),1+H_*(s)).
\end{equation*}
\end{definition}

Let us point out that $H_1(s)<_{(2+X)\times\mathbb N}H_2(s)$ holds for~$s\neq 0$. Indeed, the latter entails that we have $n(s)>0$, so that $H(T^0(s))=H(s)$ witnesses $H_*(s)>0$. We have included $H_*(s)$ in order to obtain the following result, which complements Lemma~\ref{lem:star-extend} (note that the corresponding statement for $i=0,1$ can fail).

\begin{lemma}[$\rca_0$]\label{lem:subterms-descend}
For any $s\in\varphi_{1+X}0$ with $N(T(s))\geq N(s)>0$ we have
\begin{equation*}
H_2(T(s)[N(s)])<_{(2+X)\times\mathbb N}H_2(s[N(s)])=H_2(s).
\end{equation*}
\end{lemma}
\begin{proof}
As in the proof of Lemma~\ref{lem:star-extend}, we have
\begin{equation*}
T_*^n(T(s))=T_*(s)\quad\text{for}\quad n=N(T(s))-N(s)+1.
\end{equation*}
Note that this does even hold for $T_*(s)=0$ (which was excluded in the cited proof), since the latter entails $N(s)\leq 1$. In view of $N(T(s))\geq N(s)$, we get
\begin{equation*}
T(s)[N(s)]=T_*^{N(T(s))-N(s)}(T(s))=T_*^{n-1}(T(s))=:t.
\end{equation*}
Since $T_*$ is constructed by iterating~$T$, we can write $t=T^k(s)$ with~$k>0$. We get
\begin{equation*}
T^{k+n(t)}(s)=T^{n(t)}(t)=T_*(t)=T_*^n(T(s))=T_*(s).
\end{equation*}
The assumption $N(s)>0$ ensures $n-1<N(T(s))$ and hence $t\neq 0$. The latter implies~$n(t)>0$. For $i<k+n(t)$ we can also conclude that $T^{i+1}(s)$ is a proper subterm of $T^i(s)\neq 0$. It follows that $n=k+n(t)$ is minimal with $T^n(s)=T_*(s)$, which entails $k+n(t)=n(s)$. Now the minimality of $n(s)$ yields the inequality
\begin{equation*}
H(T(s)[N(s)])=H(T^k(s))\leq_{2\times X}H(s)
\end{equation*}
between the first components of the values of~$H_2$. If this inequality is strict, then we are done. Now assume that $H(s)$ is equal to $H(T(s)[N(s)])=H(t)=H(T^k(s))$. Then any $i<n(t)$ with $H(T^i(t))=H(t)$ corresponds to a number $k+i<n(s)$ with $H(T^{k+i}(s))=H(T^i(t))=H(s)$. The number $0<k$ provides an additional exponent with $H(T^0(s))=H(s)$. We thus get $H_*(s)>H_*(t)=H_*(T(s)[N(s)])$, which is the required inequality in the second components of~$H_2$.
\end{proof}

The following monotonicity property will also be required:

\begin{lemma}[$\rca_0$]\label{lem:H-star-monotone}
Consider terms $s,t\in\varphi_{1+X}0$. If we have $T_*(s)=T_*(t)$, $H(s)=H(t)$ and $s\leq_{\varphi_{1+X}0}t$, then we have $H_*(s)\leq H_*(t)$.
\end{lemma}
\begin{proof}
We distinguish cases according to the value of $H(s)=H(t)$. First assume that the latter is equal to the minimal element $-1\in 2+X$. In case~$s=0$ we have $H_*(s)=0$, and the conclusion of the lemma is trivial. Now consider an inequality
\begin{equation*}
s=\varphi_{x_0}s_0+\dots+\varphi_{x_{m-1}}s_{m-1}<_{\varphi_{1+X}0}\varphi_{y_0}t_0+\dots+\varphi_{y_{n-1}}t_{n-1}=t.
\end{equation*}
Let $i<m$ and $k<n$ be maximal with $\varphi_{x_i}s_i=\varphi_{x_0}s_0$ and $\varphi_{y_k}t_k=\varphi_{y_0}t_0$, respectively. By induction on $j\leq i$ we obtain $T^j(s)=\varphi_{x_j}s_j+\dots+\varphi_{x_{m-1}}s_{m-1}$. If we have $i<m-1$, then $T^i(s)$ is a sum and we have $T^{i+1}(s)=\varphi_{x_i}s_i$, so that
\begin{equation*}
-1=H(T^0(s))=\dots=H(T^i(s))<_{2+X}H(T^{i+1}(s))=x_i\in 1+X.
\end{equation*}
In case $i=m-1$ we have $T^i(s)=\varphi_{x_i}s_i$ while all previous iterates are sums. These considerations (and analogous ones for~$t$) show that we have $T_*(s)=\varphi_{x_i}s_i=\varphi_{x_0}s_0$ and $T_*(t)=\varphi_{y_0}t_0$, as well as
\begin{equation*}
H_*(s)=\begin{cases}
i+1 & \text{if $i<m-1$},\\
i & \text{if $i=m-1$},
\end{cases}
\qquad\text{and}\qquad
H_*(t)=\begin{cases}
k+1 & \text{if $k<n-1$},\\
k & \text{if $k=n-1$}.
\end{cases}
\end{equation*}
Hence the assumption $T_*(s)=T_*(t)$ amounts to $\varphi_{x_0}s_0=\varphi_{y_0}t_0$. Recall that summands are always in weakly decreasing order, and that sums are compared lexicographically. From $s<_{\varphi_{1+X}0}t$ we can thus infer $i\leq k$ and $i<n-1$. This readily implies $H_*(s)\leq H_*(t)$, as desired. Let us now consider the case where $H(s)=H(t)=:x$ is different from~$-1$. Assuming $T_*(s)=T_*(t)$ and $H_*(t)<H_*(s)$, we will show~$t<_{\varphi_{1+X}0}s$. For $i\leq n(s)$, define
\begin{equation*}
H_*(s,i)=\left|\{n\,|\,i\leq n<n(s)\text{ and }H(T^n(s))=x=H(s)\}\right|.
\end{equation*}
Let $H_*(t,j)$ for $j\leq n(t)$ be defined in the same way. By induction from $j=n(t)$ down to $j=0$ we will show the following: For all $i\leq n(s)$ with $H_*(t,j)<H_*(s,i)$ we have $T^j(t)<_{\varphi_{1+X}0}T^i(s)$. In view of $H_*(t,0)=H_*(t)<H_*(s)=H_*(s,0)$, the case $j=0=i$ will yield $t=T^0(t)<_{\varphi_{1+X}0}T^0(s)=s$, as desired. As preparation for  the inductive argument, we note that $T(r)<_{\varphi_{1+X}0}r$ holds for any~$r\neq 0$ (the crucial case $r=\varphi_yr'$ is treated in~\cite[Lemma~3.1]{freund_predicative-collapsing}). Iteratively, we get $T^n(r)<_{\varphi_{1+X}0}T^m(r)$ for $m<n\leq n(r)$. Concerning the base $j=n(t)$ of our induction, we observe that the assumption $H_*(t,n(t))<H_*(s,i)$ entails $i<n(s)$, as $H_*(s,n(s))=0$. We get
\begin{equation*}
T^{n(t)}(t)=T_*(t)=T_*(s)=T^{n(s)}(s)<_{\varphi_{1+X}0}T^i(s).
\end{equation*}
In the induction step, we consider $j<n(t)$ and $i\leq n(s)$ with $H_*(t,j)<H_*(s,i)$. Let $k\geq i$ be minimal with $H(T^k(s))=x$. We then have $H_*(s,k)=H_*(s,i)$. In view of $T^k(s)\leq_{\varphi_{1+X}0}T^i(s)$ it suffices to establish~$T^j(t)<_{\varphi_{1+X}0}T^k(s)$. The point is that $H(T^k(s))=x$ yields $T^k(s)=\varphi_x T^{k+1}(s)$. We distinguish several cases: First assume that $T^j(t)$ is of the form
\begin{equation*}
T^j(t)=\varphi_{y_0}t_0+\dots+\varphi_{y_{m-1}}t_{m-1}.
\end{equation*}
Using the induction hypothesis, we get
\begin{equation*}
\varphi_{y_0}t_0\leq_{\varphi_{1+X}0} T^{j+1}(t)<_{\varphi_{1+X}0}T^k(s)=\varphi_x T^{k+1}(s).
\end{equation*}
Since the term on the right behaves like an additively principal ordinal (cf.~\cite[Section~2]{rathjen-weiermann-atr}), we obtain $T^j(t)<_{\varphi_{1+X}0}T^k(s)$, as needed. Now consider a term
\begin{equation*}
T^j(t)=\varphi_yT^{j+1}(t)\quad\text{with}\quad y=H(T^j(t))\leq_{2+X}H(t)=x.
\end{equation*}
For $y<_{\varphi_{1+X}0}x$ we observe $T^{j+1}(t)<_{\varphi_{1+X}0}\varphi_x T^{k+1}(s)$ as above. In the text before Remark~\ref{rmk:chi-semantics} we have explained that the order on $\varphi_{1+X}0$ is determined by a certain equivalence~(\ref{eq:ineq-Veblen}). The latter yields $T^j(t)=\varphi_yT^{j+1}(t)<_{\varphi_{1+X}0}\varphi_x T^{k+1}(s)=T^k(s)$, as needed. Finally, assume $y=x$. We then have $H(T^j(t))=x$ and hence
\begin{equation*}
H_*(t,j+1)<H_*(t,j)\leq H_*(s,k)-1=H_*(s,k+1).
\end{equation*}
By induction hypothesis we get $T^{j+1}(t)<_{\varphi_{1+X}0}T^{k+1}(s)$. In view of $y=x$, equivalence~(\ref{eq:ineq-Veblen}) yields $T^j(t)=\varphi_yT^{j+1}(t)<_{\varphi_{1+X}0}\varphi_x T^{k+1}(s)=T^k(s)$ once again.
\end{proof}

In order to describe the desired embedding of $\varphi_{1+X}0$ into $\widehat A((2+X)\times\mathbb N)$, we introduce some final notation: Given an order~$Y$, elements $y_0,\dots,y_{k-1}\in Y$, terms $s_0,\dots,s_{k-1}\in\widehat A(Y)$, and a further term $t=\chi^{t_{n-1}}_{z_{n-1}}\circ\dots\circ\chi^{t_0}_{z_0}(1)\in\widehat A(Y)$ of the indicated form (i.\,e.~with $t\neq 0$), we write
\begin{equation*}
\chi^{s_{k-1}}_{y_{k-1}}\circ\dots\circ\chi^{s_0}_{y_0}(t):=\chi^{s_{k-1}}_{y_{k-1}}\circ\dots\circ\chi^{s_0}_{y_0}\circ\chi^{t_{n-1}}_{z_{n-1}}\circ\dots\circ\chi^{t_0}_{z_0}(1).
\end{equation*}
To ensure that this expression is a term in $\widehat A(Y)$, one needs to check that we have $y_{k-1}<_Y\dots<_Yy_0<_Yz_{n-1}$ (where the last inequality is dropped in case $n=0$) and $s_i<_{\widehat A(Y)}\chi^{s_{i-1}}_{y_{i-1}}\circ\dots\circ\chi^{s_0}_{y_0}(t)$ for all $i<k$ (cf.~Definition~\ref{def:hat-A-objects}). If $s\neq 0$, then $R(s)$, $T(s)$ and $T_*(s)$ are proper subterms of~$s$, as observed before. Hence the following definition amounts to a recursion over terms. As part of Proposition~\ref{prop:o-into-A} below, we will verify that the recursion does indeed yield values in $\widehat A((2+X)\times\mathbb N)$.

\begin{definition}[$\rca_0$]\label{def:emb-phi-A}
The function $o:\varphi_{1+X}0\to\widehat A((2+X)\times\mathbb N)$ is given by the recursive clauses $o(0)=1$ and
\begin{equation*}
o(s)=\chi^{o(R(s))}_{H_0(s)}\circ\chi^{o(T(s))}_{H_1(s)}\circ\chi^0_{H_2(s)}(o(T_*(s)))\quad\text{for $s\neq 0$}.
\end{equation*}
\end{definition}

In the following we will often abbreviate $\chi^{o(R(s))}_{H_0(s)}\circ\chi^{o(T(s))}_{H_1(s)}\circ\chi^0_{H_2(s)}$ by $\chi^s$, so that the recursive clause for $s\neq 0$ can be written as $o(s)=\chi^s(o(T_*(s)))$. Let us observe that $s\neq 0$ entails $N(s)>0$ and $T_*(s)[i]=s[i]$ for $i\leq N(T_*(s))=N(s)-1$. Together with $s=s[N(s)]$, we recursively get
\begin{equation*}
o(s)=\chi^{s[N(s)]}\circ\dots\circ\chi^{s[1]}(1).
\end{equation*}
In particular, this does cover $o(0)=1$, since we have $N(0)=0$.

\begin{proposition}[$\rca_0$]\label{prop:o-into-A}
The recursion from Definition~\ref{def:emb-phi-A} yields an embedding of $\varphi_{1+X}0$ into $\widehat A((2+X)\times\mathbb N)$, for any linear order~$X$.
\end{proposition}
\begin{proof}
As in previous arguments, we will need a suitable length function: Let us define $L:\varphi_{1+X}0\to\mathbb N$ by stipulating $L(0)=0$, $L(\varphi_xs)=L(s)+1$ and
\begin{equation*}
L(\varphi_{x_0}s_0+\dots+\varphi_{x_{n-1}}s_{n-1})=L(s_0)+\dots+L(s_{n-1})+2n
\end{equation*}
for $n\geq 2$. Note that this length function respects our notion of subterm: For $s\neq 0$ we can observe that $L(R(s))$, $L(T(s))$ and $L(T_*(s))$ are all strictly smaller than~$L(s)$. In view of $s[i]=T_*(s[i+1])$ we get $L(s[0])<\dots<L(s[N(s)])=L(s)$. We will simultaneously prove
\begin{align*}
r\in\varphi_{1+X}0\quad&\Rightarrow\quad o(r)\in\widehat A((2+X)\times\mathbb N),\\
s<_{\varphi_{1+X}0}t\quad&\Rightarrow\quad o(s)<_{\widehat A((2+X)\times\mathbb N)}o(t)
\end{align*}
by induction on $L(r)$ and $L(s)+L(t)$, respectively. To simplify notation, we will often omit subscripts of inequality signs. Concerning the first implication for $r\neq 0$, we consider
\begin{equation*}
o(r)=\chi^{o(R(r))}_{H_0(r)}\circ\chi^{o(T(r))}_{H_1(r)}\circ\chi^0_{H_2(r)}(o(T_*(r)))=\chi^{r[N(r)]}\circ\dots\circ\chi^{r[1]}(1).
\end{equation*}
The lower indices are strictly increasing by Lemma~\ref{lem:indices-monotone} and the paragraph after Definition~\ref{def:split-indices}. Using Lemma~\ref{lem:star-extend}, we get
\begin{equation*}
o(T(r))=\chi^{T(r)[N(T(r))]}\circ\dots\circ\chi^{T(r)[N(r)]}\circ\chi^{r[N(r)-1]}\circ\dots\circ\chi^{r[1]}(1),
\end{equation*}
where the first (resp.~second) part of the composition is omitted in case that we have $N(T(r))=N(r)-1$ (resp.~$N(r)=1$). We can deduce
\begin{equation*}
o(T(r))<\chi^0_{H_2(r)}(o(T_*(r)))=\chi^0_{H_2(r)}\circ\chi^{r[N(r)-1]}\circ\dots\circ\chi^{r[1]}(1)
\end{equation*}
as follows: If we have $N(T(r))=N(r)-1$, the expression on the right is a proper extension of the expression for $o(T(r))$ above. This yields the desired inequality, since terms in $\widehat A((2+X)\times\mathbb N)$ are compared lexicographically (cf.~Definition~\ref{def:hat-A-objects}). If we have $N(T(r))\geq N(r)$, our inequality follows from \mbox{$H_2(T(r)[N(r)])<_{(2+X)\times\mathbb N}H_2(r)$}, which holds by Lemma~\ref{lem:subterms-descend} (note that $r\neq 0$ entails $N(r)>0$). For future use we~record that we have, in particular, established
\begin{equation*}
o(T(r))<\chi^{o(R(r))}_{H_0(r)}\circ\chi^{o(T(r))}_{H_1(r)}\circ\chi^0_{H_2(r)}\circ\chi^{r[N(r)-1]}\circ\dots\circ\chi^{r[1]}(1)=o(r)
\end{equation*}
for $r\neq 0$. To conclude that we have $o(r)\in\widehat A((2+X)\times\mathbb N)$, it remains to show
\begin{equation*}
o(R(r))<_{\widehat A((2+X)\times\mathbb N)}\chi^{o(T(r))}_{H_1(r)}\circ\chi^0_{H_2(r)}(o(T_*(r))).
\end{equation*}
We have already shown that $o(T(r))$ is strictly smaller than the right side, so that the open inequality reduces to $o(R(r))\leq o(T(r))$. The latter follows from the simul\-ta\-neous induction hypothesis, as $R(r)\leq T(r)$ and $L(R(r))+L(T(r))<L(r)$ are readily verified (note that this depends on the factor~$2$ in the definition of~$L$). To prove the second part of the induction step, we assume $s<t$ and deduce~$o(s)<o(t)$. The latter is immediate if we have $s=0\neq t$. Now consider an inequality
\begin{equation*}
s=\varphi_xs'<_{\varphi_{1+X}0}\varphi_yt'=t.
\end{equation*}
Note that we have $T(s)=s'$ and $T(t)=t'$. As above, Lemma~\ref{lem:star-extend} yields
\begin{align*}
o(s')&=\chi^{s'[N(s')]}\circ\dots\circ\chi^{s'[N(s)]}\circ\chi^{s[N(s)-1]}\circ\dots\circ\chi^{s[1]}(1),\\
o(t')&=\chi^{t'[N(t')]}\circ\dots\circ\chi^{t'[N(t)]}\circ\chi^{t[N(t)-1]}\circ\dots\circ\chi^{t[1]}(1),
\end{align*}
where $N(s')=N(s)-1$ and $N(s)=1$ are possible (analogously for~$t$). Let us distinguish cases according to equivalence~(\ref{eq:ineq-Veblen}), which can be found in the text before Remark~\ref{rmk:chi-semantics} above. First assume that we have $x>_{1+X}y$ and $s<t'$. The latter entails $o(s)<o(t')$ by induction hypothesis. In the first part of the induction step we have already established $o(t')=o(T(t))<o(t)$. We thus get $o(s)<o(t)$ by transitivity. Now assume $x<_{1+X}y$ and $s'<_{\varphi_{1+X}0}t$. By induction hypothesis we get $o(s')<o(t)$, hence in particular
\begin{equation*}
\chi^{s[N(s)-1]}\circ\dots\circ\chi^{s[1]}(1)<_{\widehat A((2+X)\times\mathbb N}o(t)=\chi^{t[N(t)]}\circ\dots\circ\chi^{t[1]}(1).
\end{equation*}
We immediately obtain $o(s)=\chi^{s[N(s)]}\circ\dots\circ\chi^{s[1]}(1)<o(t)$, unless the right side of the previous inequality is an end extension of the left side. More precisely, it remains to consider the case where we have $N(s)-1<N(t)$ and
\begin{equation}\label{eq:prev-equal}\tag{$\dag$}
\chi^{s[N(s)-1]}\circ\dots\circ\chi^{s[1]}(1)=\chi^{t[N(s)-1]}\circ\dots\circ\chi^{t[1]}(1).
\end{equation}
In this case we observe
\begin{equation*}
H(s)=x<_{2+X}y=H(t)\leq_{2+X}H(t[N(s)]),
\end{equation*}
where the last inequality relies on Lemma~\ref{lem:indices-monotone}. As desired, we can conclude
\begin{multline*}
o(s)=\chi^{o(R(s))}_{H_0(s)}\circ\chi^{o(T(s))}_{H_1(s)}\circ\chi^0_{(H(s),H_*(s))}\circ\chi^{s[N(s)-1]}\circ\dots\circ\chi^{s[1]}(1)<\\
\chi^0_{(H(t[N(s)]),H_*(t[N(s)]))}\circ\chi^{t[N(s)-1]}\circ\dots\circ\chi^{t[1]}(1)<o(t).
\end{multline*}
Finally, assume that we have $x=y$ and $s'<t'$. We first infer $o(s')<o(t')<o(t)$. As above, we get $o(s)<o(t)$ unless we have $N(s)-1<N(t)$ and equation~(\ref{eq:prev-equal}) holds. If we have $N(s)<N(t)$, then we obtain
\begin{equation*}
H(s)=x=y=H(t)<_{2+X}H(t[N(s)]),
\end{equation*}
where the inequality from Lemma~\ref{lem:indices-monotone} is now strict. We then get \mbox{$o(s)<o(t)$} as above. It remains to consider the case where we have $N(s)=N(t)$. As preparation for this case, we recall that any term $r\in\widehat A((2+X)\times\mathbb N)$ is determined by the tuple $(H(r),T(r),R(r))$, as noted before Definition~\ref{def:coeff-subterms} above. Recursively, this allows us to recover~$r$ from $o(r)$. From equation~(\ref{eq:prev-equal}) and $N(s)=N(t)$ we thus get
\begin{equation*}
T_*(s)=s[N(s)-1]=t[N(s)-1]=t[N(t)-1]=T_*(t).
\end{equation*}
By Lemma~\ref{lem:H-star-monotone} we obtain~$H_*(s)\leq H_*(t)$. Together with the inequality $o(s')<o(t')$ from the induction hypothesis, we can finally infer
\begin{multline*}
o(s)=\chi^1_{(H(s),0)}\circ\chi^{o(s')}_{(H(s),1)}\circ\chi^0_{(H(s),1+H_*(s))}(o(T_*(s)))<\\
\chi^1_{(H(t),0)}\circ\chi^{o(t')}_{(H(t),1)}\circ\chi^0_{(H(t),1+H_*(t))}(o(T_*(t)))=o(t).
\end{multline*}
In the case of an inequality
\begin{equation*}
s=\varphi_{x_0}s_0+\dots+\varphi_{x_{m-1}}s_{m-1}<_{\varphi_{1+X}0}\varphi_yt'=t
\end{equation*}
we have $H(s)=-1<_{2+X}y=H(t)$. Due to this fact, the argument is similar to the one for $s=\varphi_xs'<\varphi_yt'=t$ with $x<_{1+X}y$. Now consider an inequality
\begin{equation*}
s=\varphi_xs'<_{\varphi_{1+X}0}\varphi_{y_0}t_0+\dots+\varphi_{y_{n-1}}t_{n-1}=t.
\end{equation*}
We then have $s\leq\varphi_{x_0}y_0\leq T(t)$ (cf.~\cite[Section~2]{rathjen-weiermann-atr}). By induction hypothesis and the above, we get $o(s)\leq o(T(t))<o(t)$. Finally, we look at
\begin{equation*}
s=\varphi_{x_0}s_0+\dots+\varphi_{x_{m-1}}s_{m-1}<_{\varphi_{1+X}0}\varphi_{y_0}t_0+\dots+\varphi_{y_{n-1}}t_{n-1}=t.
\end{equation*}
One readily verifies $T(s)\leq T(t)$. As in the argument for $s=\varphi_xs'<\varphi_xt'=t$ (where $T(s)=s'$ and $T(t)=t'$), we can now reduce to the case of $T_*(s)=T_*(t)$. In view of $H(s)=-1=H(t)$, we get $H_*(s)\leq H_*(t)$ by Lemma~\ref{lem:H-star-monotone}. It is straightforward to show that $T(s)=T(t)$ entails $R(s)<R(t)$. The induction hypothesis yields~$o(T(s))\leq o(T(t))$, and in case of equality we get~\mbox{$o(R(s))<o(R(t))$}. Similarly to the above, a lexicographic comparison in $\widehat A((2+X)\times\mathbb N)$ does now yield
\begin{multline*}
o(s)=\chi^{o(R(s))}_{(H(s),0)}\circ\chi^{o(T(s))}_{(H(s),1)}\circ\chi^0_{(H(s),H_*(s))}(o(T_*(s)))<_{\widehat A((2+X)\times\mathbb N)}\\
\chi^{o(R(t))}_{(H(t),0)}\circ\chi^{o(T(t))}_{(H(t),1)}\circ\chi^0_{(H(t),H_*(t))}(o(T_*(t)))=o(t),
\end{multline*}
which completes the step of our simultaneous induction.
\end{proof}

The following result was already stated in the introduction. Let us point out that the strengthened base theory $\rca_0+\isigma_2$ is only needed to handle the functor~$\overline A$, which is based on the Ackermann function. The corresponding result with $\widehat A$ at the place of $\overline A$ holds over $\rca_0$, as the proof will reveal.

\begin{theorem}[$\rca_0+\isigma_2$]\label{thm:Ackermann-Veblen}
For any linear order~$X$, there is an order embedding of $\overline A(X)$ into $\varphi_{1+X}0$ and an order embedding of $\varphi_{1+X}0$ into $\overline A((2+X)\times\mathbb N)$.
\end{theorem}
\begin{proof}
Due to Proposition~\ref{prop:iso-A-hat-bar}, we may replace $\overline A(Y)$ by the isomorphic order $\widehat A(Y)$, for any order~$Y$. Proposition~\ref{prop:A-into-phi} provides an embedding of $\widehat A(X)$ into~$\varphi_{1+X}0$. The latter can be embedded into $\widehat A((2+X)\times\mathbb N)$, by Proposition~\ref{prop:o-into-A}.
\end{proof}

Putting things together, we deduce the remaining claim from the introduction:

\begin{theorem}\label{thm:Goodstein-ATR}
The following are equivalent over~$\rca_0+\isigma_2$:
\begin{enumerate}[label=(\roman*)]
\item arithmetical transfinite recursion,
\item when $X$ is a well order, then so is $\overline A(X)$ (where $\overline A:\lo\to\lo$ extends the Goodstein dilator~$A:\nat\to\lo$ based on the Ackermann function),
\item the extended Goodstein theorem for the Goodstein dilator~$A$: for any Goodstein system $(b,c)$ and any $m\in A(b(0))$ there is an $i\in\mathbb N$ with $G_{b,c,m}^A(i)=0$.
\end{enumerate}
\end{theorem}
\begin{proof}
We first justify an implicit claim: Proposition~\ref{prop:Ack-Goodstein-dil} tells us that $A$ is indeed a Goodstein dilator. The equivalence between~(ii) and~(iii) holds by Theorem~\ref{thm:Goodstein-dilator}. Let us also consider the following statement:
\begin{enumerate}[label=(\roman*)]\setcounter{enumi}{3}
\item when $X$ is a well order, then so is $\varphi_{1+X}0$.
\end{enumerate}
Harvey Friedman has shown that~(i) and~(iv) are equivalent (see~\cite{rathjen-weiermann-atr,marcone-montalban} for published proofs). So it remains to establish an equivalence between~(ii) and~(iv). To avoid confusion, we recall that the two obvious definitions of well order (in terms of descending sequences and minimal elements) are equivalent over~$\rca_0$ (see e.\,g.~\cite[Lemma~2.3.12]{freund-thesis} for a detailed proof). Let us now assume~(ii) and deduce~(iv). Aiming at the latter, we consider an arbitrary well order~$X$. It is straightforward to show that $(2+X)\times\mathbb N$ is also a well order. By~(ii) we can infer that $\overline A((2+X)\times\mathbb N)$ has the same property. Due to the embedding from Theorem~\ref{thm:Ackermann-Veblen}, any descending sequence in~$\varphi_{1+X}0$ could be transformed into one in~$\overline A((2+X)\times\mathbb N)$. Hence $\varphi_{1+X}0$ must be well founded, as needed for~(iv). To show that~(iv) implies~(ii) one argues similarly, using the fact that~$\overline A(X)$ embeds into~$\varphi_{1+X}0$, also by Theorem~\ref{thm:Ackermann-Veblen}.
\end{proof}

\appendix

\section{Preservation of pullbacks in terms of support functions}\label{appendix:pullback-support}

In this appendix we provide a proof of Proposition~\ref{prop:pullback-nat-transf} from the main text (reproduced below). The result is essentially implicit in~\cite[Theorem~2.3.12]{girard-pi2} (see also~\cite[Remark~2.2.2]{freund-thesis}, which is closer to our notation). However, the setting is somewhat different: In the given references, one considers functors on arbitrary linear orders, and the support functions ensure preservation of pullbacks as well as direct limit. Our Goodstein dilators only take finite orders as input, so that limits are trivial. The proposition confirms that the condition on support functions remains the same.

{
\def\thetheorem{\ref{prop:pullback-nat-transf}}
\addtocounter{theorem}{-1}
\begin{proposition}[$\rca_0$]
The following are equivalent for $D:\nat\to\lo$:
\begin{enumerate}[label=(\roman*)]
\item the functor~$D$ preserves pullbacks,
\item there is a natural transformation $\supp:D\Rightarrow[\cdot]^{<\omega}$ such that
\begin{equation*}
\supp_n(\sigma)\subseteq\rng(f)\quad\Rightarrow\quad\sigma\in\rng(D(f))
\end{equation*}
holds for any morphism $f:m\to n$ and any $\sigma\in D(n)$.
\end{enumerate}
If a natural transformation as in~(ii) exists, then it is unique.
\end{proposition}
}
\begin{proof}
Let us first assume~(i) and deduce~(ii): Given $\sigma\in D(n)$, let $n_\sigma\leq n$ be minimal such that $\sigma\in\rng(D(e_\sigma))$ holds for some morphism $e_\sigma:n_\sigma\to n$ (where existence is witnessed by the identity on $n_\sigma=n$). To show that $e_\sigma$ is uniquely determined, we consider a competitor $e_\sigma':n_\sigma\to n$ with $\sigma\in\rng(D(e_\sigma'))$. Consider the pullback diagram
\begin{equation*}
\begin{tikzcd}[row sep=large]
m\ar[r,"f"]\ar[d,"f'",swap]\arrow[rd, phantom,"\scalebox{1.5}{$\lrcorner$}", very near start] & n_\sigma\ar[d,"e_\sigma"]\\
n_\sigma\ar[r,"e_\sigma'",swap] & n
\end{tikzcd}
\end{equation*}
in the category $\nat$, where $m=\{0,\dots,m-1\}$ is isomorphic to $\rng(e_\sigma)\cap\rng(e_\sigma')$. Note that the embedding $D(e_\sigma):D(n_\sigma)\to D(n)$ is an isomorphism onto its range. We thus get $g$ and $g'$ as in the following diagram, where  $D(e_\sigma)\circ g$ and $D(e_\sigma')\circ g'$ are the inclusion: 
\begin{equation*}
\begin{tikzcd}[row sep=large]
\rng(D(e_\sigma))\cap\rng(D(e_\sigma'))\ar[rd,"h",dashed]\ar[rrd,"g",bend left=30]\ar[rdd,"g'",swap,bend right=30] & & \\
& D(m)\ar[r,"D(f)"]\ar[d,"D(f')",swap]\arrow[rd, phantom,"\scalebox{1.5}{$\lrcorner$}", very near start] & D(n_\sigma)\ar[d,"D(e_\sigma)"]\\
& D(n_\sigma)\ar[r,"D(e_\sigma')",swap] & D(n)
\end{tikzcd}
\end{equation*}
The morphism~$h$ exists since $D$ preserves pullbacks by~(i). We now obtain
\begin{equation*}
\sigma=D(e_\sigma)\circ g(\sigma)=D(e_\sigma)\circ D(f)\circ h(\sigma)\in\rng(D(e_\sigma\circ f)).
\end{equation*}
By the minimality of $n_\sigma$, this forces~$m=n_\sigma$. It follows that $f=f'$ is the identity (note that the morphisms $f,f'$ must be strictly increasing). Since the diagram above commutes, this entails $e_\sigma=e_\sigma'$, as desired. We now show that~(ii) holds for
\begin{equation*}
\supp_n(\sigma):=\rng(e_\sigma).
\end{equation*}
If $\supp_n(\sigma)\subseteq\rng(f)$ holds for $f:m\to n$, we get $e_\sigma=f\circ e$ for some $e:n_\sigma\to m$. This yields $\sigma\in\rng(D(e_\sigma))\subseteq\rng(D(f))$, as required. It remains to show that the family of functions $\supp_n:D(n)\to[n]^{<\omega}$ is natural. For this purpose, we consider an element $\sigma\in D(n)$ and a morphism $g:n\to k$. Let us write
\begin{align*}
\sigma&=D(e_\sigma)(\sigma_0) && \text{with }e_\sigma:n_\sigma\to n,\\
D(g)(\sigma)&=D(e_{D(g)(\sigma)})(\sigma_1) && \text{with }e_{D(g)(\sigma)}:k_{D(g)(\sigma)}\to k,
\end{align*}
where $n_\sigma$ and $k_{D(g)(\sigma)}$ are as small as possible. Now consider the pullback
\begin{equation*}
\begin{tikzcd}[row sep=large]
m\ar[r,"f"]\ar[d,"f'",swap]\arrow[rd, phantom,"\scalebox{1.5}{$\lrcorner$}", very near start] & k_{D(g)(\sigma)}\ar[d,"e_{D(g)(\sigma)}"]\\
n_\sigma\ar[r,"g\circ e_\sigma",swap] & k
\end{tikzcd}
\end{equation*}
in the category~$\nat$. Similarly to the above, the minimality of $k_{D(g)(\sigma)}$ entails that $f$ must be the identity. Since the diagram commutes, we get $e_{D(g)(\sigma)}=g\circ e_\sigma\circ f'$ and thus
\begin{equation*}
D(g\circ e_\sigma\circ f')(\sigma_1)=D(e_{D(g)(\sigma)})(\sigma_1)=D(g)(\sigma),
\end{equation*}
which allows us to conclude $D(e_\sigma\circ f')(\sigma_1)=\sigma$ (note that the morphism $D(g)$ is an embedding and hence injective). Now the minimality of $n_\sigma$ ensures that $f'$ is the identity as well (on $n_\sigma=m=k_{D(g)(\sigma)}$). We thus have $e_{D(g)(\sigma)}=g\circ e_\sigma$ and hence
\begin{equation*}
\supp_k\circ D(g)(\sigma)=\rng(e_{D(g)(\sigma)})=\rng(g\circ e_\sigma)=[g]^{<\omega}(\rng(e_\sigma))=[g]^{<\omega}\circ\supp_n(\sigma),
\end{equation*}
as needed for naturality. To see that the sets $\supp_n(\sigma)$ are uniquely determined, we again write $\sigma=D(e_\sigma)(\sigma_0)$ with $e_\sigma:n_\sigma\to n$, where $n_\sigma$ is as small as possible. Due to the latter, we have $\sigma_0\notin\rng(D(f))$ for any $f:m\to n_\sigma$ with $m<n_\sigma$. Hence the implication in~(ii) can only hold if we have $\supp_{n_\sigma}(\sigma_0)=n_\sigma=\{0,\dots,n_\sigma-1\}$. Together with naturality, this forces
\begin{equation*}
\supp_n(\sigma)=\supp_n\circ D(e_\sigma)(\sigma_0)=[e_\sigma]^{<\omega}\circ\supp_{n_\sigma}(\sigma_0)=[e_\sigma]^{<\omega}(n_\sigma)=\rng(e_\sigma).
\end{equation*}
Let us now assume~(ii) and deduce~(i): Consider a pullback
\begin{equation*}
\begin{tikzcd}[row sep=large]
l\ar[r,"f"]\ar[d,"f'",swap]\arrow[rd, phantom,"\scalebox{1.5}{$\lrcorner$}", very near start] & m\ar[d,"e"]\\
m'\ar[r,"e'",swap] & n
\end{tikzcd}
\end{equation*}
in the category~$\nat$. We need to show that any commutative diagram of the form
\begin{equation*}
\begin{tikzcd}[row sep=large]
X\ar[rrd,"g",bend left=30]\ar[rdd,"g'",swap,bend right=30] & & \\
& D(l)\ar[r,"D(f)"]\ar[d,"D(f')",swap] & D(m)\ar[d,"D(e)"]\\
& D(m')\ar[r,"D(e')",swap] & D(n)
\end{tikzcd}
\end{equation*}
can be completed by a unique morphism $h:X\to D(l)$. Note that it is enough to construct $h$ as a function: the latter will automatically be an order embedding, since the same holds for the morphisms $D(f)$ and $g$. Also, uniqueness is immediate since the embedding $D(f)$ is injective. In order to find the required value $h(x)\in D(l)$ for $x\in X$, we observe
\begin{equation*}
\supp_n(D(e)\circ g(x))=[e]^{<\omega}\circ\supp_m(g(x))\subseteq\rng(e)
\end{equation*}
and similarly $\supp_n(D(e)\circ g(x))\subseteq\rng(e')$. The given pullback in~$\nat$ yields a function $h_0$ as in the following diagram, where $j$ and $j'$ are defined by the condition that $\rng(e\circ j)=\rng(e'\circ j')$ is equal to the set $\rng(e)\cap\rng(e')$ of cardinality~$k$:
\begin{equation*}
\begin{tikzcd}[row sep=large]
k\ar[rd,"h_0",dashed]\ar[rrd,"j",bend left=30]\ar[rdd,"j'",swap,bend right=30] & & \\
& l\ar[r,"f"]\ar[d,"f'",swap]\arrow[rd, phantom,"\scalebox{1.5}{$\lrcorner$}", very near start] & m\ar[d,"e"]\\
& m'\ar[r,"e'",swap] & n
\end{tikzcd}
\end{equation*}
We can now observe
\begin{equation*}
\supp_n(D(e)\circ g(x))\subseteq\rng(e)\cap\rng(e')=\rng(e\circ j)=\rng(e\circ f\circ h_0)\subseteq\rng(e\circ f).
\end{equation*}
By the implication in~(ii), we obtain a (necessarily unique) element $\sigma\in D(l)$ with $D(e)\circ g(x)=D(e\circ f)(\sigma)$. The latter entails $g(x)=D(f)(\sigma)$, so that $h(x)=\sigma$ is the required value.
\end{proof}

\bibliographystyle{amsplain}
\bibliography{Ackermann-Goodstein-functorial}

\end{document}